\RequirePackage{fix-cm}
\documentclass[smallextended,envcountsect]{svjour3}       
\smartqed  
\usepackage{verbatim,subfigure,wrapfig,sidecap,float,booktabs,array}
\usepackage{palatino,url}
\usepackage{amsmath,amssymb,mathtools}
\usepackage[all]{xy}
\usepackage[pdftex]{graphicx}
\usepackage[pdftex,colorlinks,urlcolor=blue]{hyperref}
\usepackage{enumerate,framed}
\usepackage[final]{pdfpages}
\usepackage{ifthen,epstopdf}
\usepackage{soul} 
\usepackage{color, colortbl}
\usepackage{hhline} 
\usepackage[normalem]{ulem} 
\usepackage{cancel} 

\newcommand{\mbf}[1]{\mathbf{#1}}
\newcommand{\inR}[2]{
	\ifthenelse{\equal{#2}{1}}{
		#1 \in \mathbb{R}
	}{
	\mathbf{#1} \in \mathbb{R}^{#2}
}}
\newcommand{\lr}[2]{
	\ifthenelse{\equal{#1}{()}}{\left( #2 \right)}{}
	\ifthenelse{\equal{#1}{[]}}{\left[ #2 \right]}{}
	\ifthenelse{\equal{#1}{\|\|}}{\left\| #2 \right\|}{}
	\ifthenelse{\equal{#1}{\{\}}}{\left\{ #2 \right\}}{}
	\ifthenelse{\equal{#1}{||}}{\left| #2 \right|}{}
	\ifthenelse{\equal{#1}{<>}}{\left\langle #2 \right\rangle}{}
}
\newcommand{\RR}[3]{
\ifthenelse{\equal{#3}{1}}{
#1 : \mathbb{R}^{#2} \rightarrow \mathbb{R}
}{
#1 : \mathbb{R}^{#2} \rightarrow \mathbb{R}^{#3}
}}

\DeclareMathOperator*{\argmax}{argmax}
\DeclareMathOperator*{\argmin}{argmin}
\begin{document}

\title{The Sparse Principal Component Analysis Problem: Optimality Conditions and Algorithms
}


\author{Amir Beck         \and
        Yakov Vaisbourd \\
}


\institute{Amir Beck (correspoding author) \at
              Faculty of Industrial Engineering and Management, Technion, Haifa, Israel \\
              \email{becka@ie.technion.ac.il}           
           \and
           Yakov Vaisbourd \at
               Faculty of Industrial Engineering and Management, Technion, Haifa, Israel \\
               \email{yakovv@campus.technion.ac.il}
}

\date{Received: date / Accepted: date}

\maketitle

\begin{abstract}
  Sparse principal component analysis addresses the problem of finding a linear combination of the variables in a given data set with a sparse coefficients vector that maximizes the variability of the data. This model enhances the ability to interpret the principal components, and is applicable in a wide variety of fields including genetics and finance, just to name a few.

  We suggest a necessary coordinate-wise-based optimality condition, and show its superiority over the stationarity-based condition that is commonly used in the literature, and which is the basis for many of the algorithms designed to solve the problem. We devise algorithms that are based on the new optimality condition, and provide numerical experiments that support our assertion that algorithms, which are guaranteed to converge to stronger optimality conditions, perform better than algorithms that converge to points satisfying weaker optimality conditions.

\end{abstract}

\section{Introduction}

\label{intro}
Principal component analysis (PCA) is a well known data-analytic technique that linearly transforms a given set of data to some equivalent representation. This transformation is defined in such a manner that any variable in the new representation, called a \emph{principal component} (PC), expresses most of the variance in the data, which is not expressed by the PCs that precede it. The linear combination defining each of the PCs is given by a coefficients (also termed \emph{loadings}) vector. In terms of the covariance (or correlation) matrix of the data, the coefficients vector of the \emph{k}-th PC is the eigenvector that corresponds to the \emph{k}-th largest eigenvalue \cite{1-1}.
One major drawback of PCA is that commonly the coefficients vectors are dense, i.e., each PC is a linear combination of much, if not most, of the original variables, which causes a difficulty in interpreting the obtained PCs. This disadvantage encouraged a wide interest in the sparsity constrained version of PCA, which imposes an additional constraint, enforcing the coefficients vector not to exceed some predetermined sparsity level \emph{s}.\\
\indent Enforcing sparsity on the coefficients vector is commonly acceptable in some applications. For example, in the exploration of micro-array gene expression patterns, PCA is employed in order to classify different tissues according to their gene expression. It is also desirable that such discrimination can be executed by utilizing only a small subset of the genes, thus encouraging sparse solutions \cite{3-1_40_Misra2002}. The desire to obtain interpretable coefficients vectors is not the only reason to favor the sparse PCA model. For example, some financial applications will prefer sparse solutions in order to reduce transaction costs \cite{8-1}. Clearly, incorporating an additional sparsity constraint will provide a PC that, generally, does not explain all of the variance which is explained by the regular PC; nevertheless, in such applications, this sacrifice is acceptable with respect to the obtained benefits.
We refer to this formulation as the sparsity constrained formulation, and it is merely one of several alternative formulations considered in the literature. The common alternatives are the result of treating the sparsity term, or its relaxation, by a penalty approach.
The sparse PCA problem is a difficult non-convex problem, and can be optimally solved only for small scale problems by performing exhaustive or a branch and bound search over all possible support sets \cite{3-1_24}. Thus, in order to handle large scale problems, the algorithms proposed in the literature are seeking to find an approximate solution. One of the first methods, suggested by Cadima and Jolliffe \cite{3-1_7}, is to threshold the smallest, in absolute value, elements of the dominant eigenvector. Unfortunately, this remarkably simple approach is known to frequently provide poor results. In \cite{3-1_24} Moghaddam et al. proposed several greedy methods. An advantage of these methods is that they produce a full path of solutions (i.e., a solution for each of the values of sparsity level up to \emph{s}), but the necessity to perform a large amount of eigenvalue computations at each step render them quite computationally expensive. In \cite{3-1_10}, d'Aspremont et al. proposed an approximate greedy approach that obviates the necessity to perform most of the eigenvalue computations by evaluating a lower bound on the eigenvalues, which results in a substantial reduction of computation time. Another approach presented by d'Aspremont et al. in \cite{3-1_10}, and earlier in \cite{3-1_11}, is to consider a semidefinite programming formulation with a rank constraint for some of the relaxed and/or penalized models of PCA. These equivalent formulations are still hard non-convex problems, and thus a relaxed model is solved and an approximate solution is derived for the original problem. The algorithms used to solve the SDP relaxations are not applicable for large scale problems, rendering this approach as non-scalable. In \cite{3-1_17}, encouraged by the LASSO approach suggested for regression \cite{3-1_32}, Jolliffe et al. proposed the absolute value norm constrained formulation under the name SCoTLass (simplified component technique LASSO), which is a relaxation of  the sparsity constrained problem. In practice, the numerical study was conduced on the penalized version by implementing the projected gradient algorithm. The relaxed model was further considered in the literature. An alternating minimization scheme to solve the constrained formulation was proposed in \cite{3-1_35}. Another work that addressed the constrained formulation was motivated by the expectation maximization algorithm for probabilistic PCA \cite{3-1_29}. Even though the work addressed the constrained formulation, the sequence generated by the method in \cite{3-1_29} is guaranteed to be \emph{s}-sparse. Penalized versions were also considered extensively. In \cite{3-1_40} Zou et al. formulated the sparse PCA as a regression-type model, where the \emph{i}-th principal component was approximated by the linear combination of the original variables. A LASSO and ridge penalties are imposed on the coefficients vector forming the elastic net model that generalizes the LASSO \cite{3-1_40_ZouHastie2005} and an alternating minimization algorithm, called SPCA, was proposed. In \cite{3-1_28} Shen and Huang proposed several iterative schemes to solve the penalized versions via regularized SVD. These methods were considered further in \cite{3-1_18}, where a gradient scheme was proposed and a convergence analysis, that was missing in \cite{3-1_28}, was also provided.\\
\indent Recently, Luss and Teboulle showed in \cite{3-1} that the seemingly different methods proposed in \cite{3-1_18,3-1_28,3-1_29,3-1_30,3-1_35,3-1_40} are some particular realizations of the conditional gradient algorithm with unit step-size. The work \cite{3-1} proposed a unified algorithmic framework which they refer to as ConGradU  (the well known conditional gradient with unit step size) and established convergence results, showing that the algorithm produces a point satisfying some necessary first order optimality criteria. Some novel schemes are provided. One of them addresses directly the sparsity constrained formulation of sparse PCA.\\
\indent As already mentioned, none of the methods listed above can guarantee to produce an optimal solution. In addition,  the sparse PCA problem  does not seem to  posses a verifiable necessary and sufficient global optimality condition, and hence, in general, there is no efficient way to check if a given vector is the global optimal solution\footnote{In \cite{3-1_10} the authors suggested a sufficient optimality condition.}. Therefore, the comparison of the methods in the literature is based solely on numerical experiments without providing any theoretical justification for the advantage of a certain method over the others. However, most of the algorithms just listed will produce a solution that satisfies some necessary optimality condition. In a recent work, Beck and Eldar \cite{6-4} employed some of the aforesaid conditions in order to provide an insight regarding the success of the corresponding algorithms. Under the framework of minimizing a continuously differentiable function subject to a sparsity constraint, several necessary optimality conditions were presented. The relations between the different optimality conditions were established, showing that some of the conditions are stronger (that is, more restrictive) than others. An extension to problems over sparse symmetric sets was considered in \cite{BH14}.
In this paper, we adopt this methodology in order to establish a hierarchy between two necessary optimality conditions for  the sparsity constrained sparse PCA problem. The first condition that we consider is a  well known first order condition, that was originally presented  in the context of the sparse PCA problem in \cite{3-1}. We will refer to it as the \emph{complete (co) stationarity} condition. Much of the existing algorithms in the literature are actually guaranteed to converge to a co-stationary point. The second condition, which we call \emph{coordinate-wise (CW) maximality}, is a generalization of one of the conditions considered in \cite{6-4}, and it essentially states that the function value cannot be improved by making changes of at most two coordinates. \\
In the following section we will explicitly define the conditions under consideration. In Section \ref{sec:Nec_Opt_Cond_Hierarchy}, we will establish the relation between the conditions, showing that the CW-maximality condition is stronger (that is, more restrictive) than co-stationarity. In Section \ref{sec:Algos}, we will introduce algorithms that produce points satisfying the aforementioned conditions and finally, in Section \ref{sec:Num_Res}, we will provide a numerical study on simulated and real life data that supports our assertion that algorithms that correspond to stronger conditions are more likely to provide better results.

\section{Necessary Optimality Conditions}
\label{sec:Nec_Opt_Cond}
Throughout the paper, we consider the following sparsity constrained problem:
\begin{gather}
\max \{f(\mbf{x}):\mbf{x}\in S \},
\tag{P}
\label{P}
\end{gather}

where $f$ is a continuously differentiable convex function over $\mathbb{R}^n$ and
$$S  := \{\inR{x}{n}:\|\mbf{x}\|_2\leq1,\|\mbf{x}\|_0\leq s\},$$
with $\|\cdot\|_0$  being the so-called $l_0$-norm defined by $\|\mbf{x}\|_0 := |\lr{\{\}}{i:x_i\neq 0}|$\footnote{Note that the $l_0$-norm is not actually a norm since it does not satisfy the absolute homogeneity property.}. As a special case, when the objective function is chosen as $f(\mbf{x}) = \mbf{x}^T \mbf{A} \mbf{x}$, where $\mbf{A}$ is a given positive semidefinite matrix, problem (P) amounts to the \textit{$l_0$-constrained sparse PCA model}:
\begin{equation}
\label{SPCA}
\tag{SPCA}
\max \{\mbf{x}^T\mbf{Ax}:\mbf{x}\in S \}.
\end{equation}
In PCA applications, $\mbf{A}$ usually stands for the covariance matrix of the data.\\
\indent In this section, we will present two necessary optimality conditions for the general model \eqref{P}. Although our main motivation is to study the sparse PCA problem, we will nonetheless consider the  general model \eqref{P}, since our results are also applicable in this general setting.

Prior to presenting the optimality conditions, we will introduce in the following subsection some notation and definitions that will be used in our analysis.
\subsection{Notation and Definitions}
\label{subsec:Not_and_Def}
A subvector of a given vector $\inR{x}{n}$ corresponding  to a set of indices \mbox{$T\subseteq \lr{\{\}}{1,2,\dots,n}$} is denoted by $\mbf{x}_T$. Similarly, we will denote the subvector of the gradient $\nabla f(\mbf{x})$ corresponding to the indices in $T$ by $\nabla_Tf(\mbf{x})$. The sign of a given  $\alpha \in {\mathbb R}$ is denoted by $\text{sgn}(\alpha)$ and is equal to $1$ for $\alpha \geq 0$ and $-1$ for $\alpha<0$.
The support set of some arbitrary vector $\mbf{x}$ will be denoted by $I_1(\mbf{x})= \{i:x_i\neq0\}$ and its complement by $I_0(\mbf{x})= \{i:x_i=0\}$.
For a given vector $\inR{x}{n}$ and an integer $s\in \lr{\{\}}{1,2,\dots,n-1}$, we will define $M_s(\mbf{x})$ to be the $s$-th largest absolute value component in $\mbf{x}$. For such $\mbf{x}$ and $s$, we will define the sets $I_{>}(\mbf{x},s), I_{=}(\mbf{x},s)$ and $I_{<}(\mbf{x},s)$ as follows:
\begin{align*}
& I_{>}(\mbf{x},s) := \{i:|x_i|>M_s(\mbf{x})\}, \\
& I_{=}(\mbf{x},s) :=
\begin{cases}
\{i:|x_i|=M_s(\mbf{x})\}, & \|\mbf{x}\|_0 \geq s, \\
~~~~~~\qquad \emptyset, & \|\mbf{x}\|_0<s,
\end{cases}
\\
& I_{<}(\mbf{x},s) :=
\begin{cases}
\{i:|x_i|<M_s(\mbf{x})\}, & \|\mbf{x}\|_0\geq s,\\
\{i:x_i=0\}, & \|\mbf{x}\|_0<s.
\end{cases}
\end{align*}
We will also define the set $I_{\geq}(\mbf{x},s) := I_{>}(\mbf{x},s)\cup I_{=}(\mbf{x},s)$ and the set \linebreak $I_{\leq}(\mbf{x},s) := I_{<}(\mbf{x},s)\cup I_{=}(\mbf{x},s)$. Obviously, the sets $I_{>}(\mbf{x},s), I_{=}(\mbf{x},s)$ and $I_{<}(\mbf{x},s)$ form a partition of $\{1,2,\dots,n\}$. Furthermore, when $\|\mbf{x}\|_0<s$, we have that \linebreak $I_{>}(\mbf{x},s) = I_1(\mbf{x}), I_{=}(\mbf{x}) = \emptyset$ and $I_{<}(\mbf{x},s)=I_0(\mbf{x})$.\\
The sets defined above posses some convenient and elementary properties which are given in Lemma \ref{lma:S-QCLP} below. Since all the properties stated in the lemma are rather simple consequences of the definition of the sets $I_{>}(\mbf{x},s),$ $ I_{=}(\mbf{x},s)$, $I_{<}(\mbf{x},s)$, the proof is omitted.
\begin{lemma}
\label{lma:S-QCLP}
\begin{enumerate}
\item If $\mbf{x}\neq \mbf{0}$, then $I_\geq(\mbf{x},s)\neq \emptyset$.
\item If $|I_\geq(\mbf{x},s)|<s$ then $x_j=0$ for all $j\in I_<(\mbf{x},s)$.
\item For any $i\in I_{>}(\mbf{x},s),j \in I_{=}(\mbf{x},s)$ and $k\in I_{<}(\mbf{x},s)$, it holds that \linebreak $|x_i|>|x_j|>|x_k|$.
\end{enumerate}
\end{lemma}

We will frequently use the notation
$$
R_s(\mbf{x}) := \lr{\{\}}{T:I_>(\mbf{x},s) \subseteq T\subseteq I_\geq(\mbf{x},s),|T| = \min\{s,|I_\geq(\mbf{x},s)|\}}
$$
for the set containing all the subsets of indices corresponding to the nonzero $s$ largest in absolute value components of a given vector $\mbf{x}$.
When $\|\mbf{x}\|_0\leq s$, there are no more than $s$ nonzero elements in $\mbf{x}$, and the above definition actually amounts to   $R_s(\mbf{x})= \{ I_1(\mbf{x})\}$. However, when $\|\mbf{x}\|_0> s$, there might be more than one set of indices corresponding to the $s$ largest absolute value components of $\mbf{x}$.  For example, consider the vector $\mbf{x} = (3,2,1,1,1,0,0)^T$ and the sparsity level $s=3$. Then,
$$
R_3 (\mbf{x}) = \lr{\{\}}{\{1,2,3\},\{1,2,4\},\{1,2,5\}}.
$$
On the other hand, in the following examples, the set contains a single subset:
$$ R_3((0,-5,4,-3,2,0)^T) = \{\{2,3,4\}\}, R_3 ((0,0,4,-3,0,0)^T) = \{\{3,4\}\}.$$
The \textit{hard thresholding} operator maps a vector $\mbf{x} \in \mathbb{R}^n$ to the set of vectors that are generated by keeping the $s$ largest absolute value components of $\mbf{x}$ and setting all the others to zeros. This operator, which we denote by $H_s$, is formally defined by
$$
H_s(\mbf{x}) := \bigcup_{T\in R_s(\mbf{x})} \lr{\{\}}{\mbf{y}:\mbf{y}_T=\mbf{x}_T,\mbf{y}_{\bar{T}}=\mbf{0}}.
$$
Thus, for example,
\begin{gather*}
 \hspace*{-7cm} H_3((3,2,1,1,1,0,0)^T) = \\ \{ (3,2,1,0,0,0,0)^T,(3,2,0,1,0,0,0)^T, (3,2,0,0,1,0,0)^T \}.
\end{gather*}

\subsection{Complete (co) - Stationarity}
\label{subsec:co_stat}
The first condition that we consider was presented for the sparse PCA problem in \cite{3-1}. We refer to it as the complete (co) stationarity condition.
\begin{definition}[\bf{co-stationarity}]
Let $\mbf{x}$ be a feasible solution of \eqref{P}. Then, $\mbf{x}$ is called a co-stationary point of \eqref{P} over $S$ if and only if it satisfies:
\begin{gather*}
    \langle \nabla f(\mbf{x}),\mbf{v}-\mbf{x}\rangle\leq 0 \qquad \forall \mbf{v}\in S.
\end{gather*}
\end{definition}
This is probably the most elementary first order condition for constrained differentiable optimization problems. The work \cite{3-1} provided a unified framework for several algorithms designed to solve different formulations of sparse PCA. Actually, \cite{3-1} considered the co-stationarity condition over a general nonempty and compact set instead of $S$, and for this general case, the following proposition, which was originally established in \cite{rockafellar1970}, was recalled. This result follows from the convexity of the objective function.
\begin{proposition}
Let $f:\mathbb{R}^n\rightarrow \mathbb{R}$ be a continuous differentiable and convex function over $\mathbb{R}^n$, and let $C$ be a nonempty and compact set. If $\mbf{x}$ is a global maximum of $f$ over $C$, then $\mbf{x}$ is a co-stationary point over $C$, meaning that $\langle \nabla f(\mbf{x}), \mbf{v}-\mbf{x}\rangle \leq 0$ for any $\mbf{v}\in C$.
\end{proposition}

\subsection{CW-Maximality}
\label{subsec:cw_max}
The second necessary optimality condition that we will consider is coordinate-wise maximality. This optimality condition is in fact a type of a local optimality condition, stating that a given point $\mbf{x}$ is a minimizer over a neighbourhood consisting of all feasible points, that are different by at most two coordinates.
We will denote the corresponding neighbourhood by
\begin{gather*}
S_2(\mbf{x}) := \{\mbf{z}:\|\mbf{z}-\mbf{x}\|_0\leq 2,\mbf{z}\in S\}.
\end{gather*}
The formal definition of a CW-maximum point follows.
\begin{definition}[\bf{CW-maximum point}] \label{definition:CW}
Let $\mbf{x}$ be a feasible solution of \eqref{P}. Then, $\mbf{x}$ is called a coordinate-wise (CW) maximum point of \eqref{P} if and only if \linebreak $f(\mbf{x})\geq f(\mbf{z})$ for every $\mbf{z}\in S_2(\mbf{x})$.
\end{definition}
Obviously, CW-maximality, by its definition, is a necessary optimality condition.
\begin{proposition}
\label{prop:Opt_vs_CW}
Let $\mbf{x}$ be an optimal solution to \eqref{P}. Then, $\mbf{x}$ is an CW-maximum point.
\end{proposition}
 Instead of considering the neighbourhood $S_2(\mbf{x})$ in the definition of \linebreak CW-maximality (Definition \ref{definition:CW}), we could have alternatively considered \linebreak larger neighbourhoods consisting of vectors that differ from $\mbf{x}$ by at most $k$ coordinates for some $2 \leq k \leq s$:
\begin{gather*}
S_k(\mbf{x}) := \{\mbf{z}:\|\mbf{z}-\mbf{x}\|_0\leq k,\mbf{z}\in S\}.
\end{gather*}
 A similar optimality condition over such a neighbourhood can be defined, and clearly since $S_t(\mbf{x})\subseteq S_k(\mbf{x})$ for any $t\leq k$, considering neighbourhoods that differ by a larger amount of coordinates will result in stronger optimality conditions. Note that the amount of comparisons required in order to verify that a vector \mbox{$\mbf{x}\in \mathbb{R}^n$} with a full support ($I_1(\mbf{x})=s$) is  CW-maximal ($k=2$) is $O(s\cdot n)$, while changing the neighbourhood to $S_3$ will increase the amount of comparisons to $O(s\cdot n^2)$. Hence, considering such a stronger optimality condition has a substantial computational price. Keeping in mind that we seek scalable conditions and algorithms, we restrict the discussion to the case $k=2$.

\section{Optimality Conditions Hierarchy}
\label{sec:Nec_Opt_Cond_Hierarchy}

Our main result in this section is that CW-maximality is a stronger (that is, more restrictive) optimality condition than co-stationarity. This result also has an impact on the performance of the corresponding algorithms in the sense that, loosely speaking,  algorithms that are only guaranteed to converge to a co-stationary point are less likely to produce the optimal solution of the problem than algorithms that are guaranteed to converge to a CW-maximal point.
In Section \ref{sec:Num_Res}, we will show that the numerical results support this assertion.
\subsection{Technical Preliminaries}
\label{subsec:tech_pre}
We will begin by providing some auxiliary technical results that will be used in order to establish the main result. Lemma \ref{thrm:QCLP} is a trivial result, that follows directly from the Cauchy-Schwarz inequality (see also Lemma 4.1 in \cite{3-1}).

\begin{lemma}
\label{thrm:QCLP}
Suppose that $\mbf{0} \neq \mbf{q} \in \mathbb{R}^d$ and $\rho>0$. Then, the optimal solution of the optimization problem
\begin{equation}
\label{eq:QCLP}
\tag{QCLP}
\max_{\inR{x}{d}}\{\mbf{q}^T\mbf{x}:\|\mbf{x}\|_2 \leq \rho\},
\end{equation}
 is given by $\mbf{x}^* = \rho \frac{\mbf{q}}{\|\mbf{q}\|_2}$ with the optimal value of $\rho \|\mbf{q}\|_2$.
\end{lemma}

The following simple lemma is an extension of Proposition 4.3 from \cite{3-1}.


\begin{lemma}
\label{thrm:S-QCLP}
Assume that $\mbf{0} \neq \mbf{p}\in \mathbb{R}^n$. Then, the set of optimal solutions of the optimization problem
\begin{equation}
\max_{\inR{x}{n}} \{\mbf{p}^T\mbf{x}:\|\mbf{x}\|_0 \leq s,\|\mbf{x}\|_2 \leq 1\},
\tag{S-QCLP}
\label{eq:S-QCLP}
\end{equation}
is given by
\begin{gather*}
\mbf{X}^*(\mbf{p},s) := \lr{\{\}}{\frac{\mbf{x}}{\|\mbf{x}\|_2}: \mbf{x}\in H_s(\mbf{p})},
\end{gather*}
with the optimal value of $\|\mbf{p}_T\|_2$, where $T\in R_s(\mbf{p})$.
\end{lemma}
\begin{proof}
We can write \eqref{eq:S-QCLP} as
\begin{equation}
\label{341}
\max_{\begin{scriptsize}\begin{matrix}T\subseteq \{1,\dots,n\}\\|T|\leq s\end{matrix}\end{scriptsize}} \max_{\inR{x}{n}} \lr{\{\}}{\mbf{p}^T\mbf{x}:\|\mbf{x}\|_2\leq 1, I_1(\mbf{x})\subseteq T}.
\end{equation}
According to Lemma \ref{thrm:QCLP}, for each $T\subseteq \{1,\dots,n\}$ satisfying  $|T|\leq s$, the optimal value of the inner optimization problem is $\|\mbf{p}_T\|_2$, and if $\mbf{p}_T \neq \mbf{0}$, then a solution $\mbf{x}^*$ to the inner optimization problem is given by
\begin{equation} \label{345}
\mbf{x}^*_T=\frac{\mbf{p}_T}{\|\mbf{p}_T\|_2},~\mbf{x}^*_{\bar{T}} = \mbf{0}.
\end{equation}
 The problem (\ref{341}) thus reduces to
\begin{equation}
\label{350}
\max_{\begin{scriptsize}\begin{matrix}T\subseteq \{1,\dots,n\}\\|T|\leq s\end{matrix}\end{scriptsize}} \|\mbf{p}_T\|_2.
\end{equation}
Obviously, when $\|\mbf{p}\|_0 \geq s$, the optimal solutions of the latter problem are all the sets containing the indices of components corresponding to the $s$ largest absolute values in $\mbf{p}$, and when $\|\mbf{p}\|_0 <s$, the unique optimal solution is $I_1(\mbf{p})$. Thus, the set of all optimal solutions of (\ref{350}) is $R_s(\mbf{p})$. Noting that $\mbf{p}_T \neq \mbf{0}$ for any $T \in R_s(\mbf{p})$, we conclude that the optimal solutions of \eqref{eq:S-QCLP} are given by (\ref{345}) with $T$ being any set in $R_s(\mbf{p})$, which are exactly the members of $X^*(\mbf{p},s)$.
\qed
\end{proof}

Our final technical lemma states that, if a given vector $\tilde{\mbf{x}}$ is \textit{not} an optimal solution of the problem of maximizing a linear function over the unit norm, then there must be two indices $i \neq j$ for which the subvector $\tilde{\mbf{x}}_{\{i,j\}}$ is also \textit{not} an optimal solution for the problem restricted to the the variables $x_i,x_j$ (while fixing all the other variables). This lemma is rather simple, but will play a key role in the proof of the main result.

\begin{lemma}
\label{thrm:2-QCLP}
 Let $\mbf{q} \in \mathbb{R}^d$ and $\rho>0$. Suppose that $\tilde{\mbf{x}}$ satisfies $\|\tilde{\mbf{x}}\|_2\leq \rho$, and that it is \underline{not} an optimal solution of \eqref{eq:QCLP}. Then, there exist indices $i,j (i \neq j)$ such that $\tilde{\mbf{x}}_{\{i,j\}}$ is \underline{not} the optimal solution of

\begin{equation}
\max_{\mbf{x}_{\{i,j\}}\in \mathbb{R}^2} \left\{ \mbf{q}_{\{i,j\}}^T\mbf{x}_{\{i,j\}}:\|\mbf{x}_{\{i,j\}}\|_2 \leq
\Big(\rho^2 - \sum_{l\neq i,j}\tilde{x}_l^2\Big)^{1/2}\right\}.
\tag{$\text{2-QCLP}_{\{i,j\}}$}
\label{eq:2-QCLP}
\end{equation}
\end{lemma}

\begin{proof} Since $\tilde{\mbf{x}}$ is not the optimal solution of \eqref{eq:QCLP}, we obtain that $\mbf{q}\neq \mbf{0}$ (since otherwise, if $\mbf{q}=\mbf{0}$, all feasible points are also optimal). Thus, the set $I_1(\mbf{q})$ is nonempty. We will split the analysis into two cases.
\begin{itemize}
\item If $\|\tilde{\mbf{x}}\|_2 < \rho$, then take any $i \in I_1(\mbf{q})$ and $j\neq i$, and we can write
\begin{gather*}
\|\tilde{\mbf{x}}_{\{i,j\}}\|_2 < \lr{()}{\rho^2 - \sum_{l\neq i,j}\tilde{x}_l^2}^{1/2},
\end{gather*}
which together with $\mbf{q}_{\{i,j\}} \neq \mbf{0}$ (since $i\in I_1(\mbf{q})$) implies that $\tilde{\mbf{x}}_{\{i,j\}}$ is not the optimal solution of \eqref{eq:2-QCLP}, since we have, by Lemma \ref{thrm:QCLP}, that the constraint at the optimal solution must be active.
\item If, on the other hand, $\|\tilde{\mbf{x}}\|_2 = \rho$, then assume in contradiction that for  each $i \neq j$ the vector $\tilde{\mbf{x}}_{\{i,j\}}$ is the optimal solution of \eqref{eq:2-QCLP}. Take some $i \in I_1(\mbf{q})$.  For any $j\in I_0(\mbf{q})$, we know that $\tilde{\mbf{x}}_{\{i,j\}}$ is the optimal solution of \eqref{eq:2-QCLP} and thus, according to Lemma \ref{thrm:QCLP} (employed on the problem \eqref{eq:2-QCLP}), it must in particular satisfy $\tilde{x}_j =0$, that is, $j \in I_0(\tilde{\mbf{x}})$. To summarize,
    \begin{equation} \label{parta} \tilde{x}_j=0 \mbox{ for any } j \in I_0(\mbf{q}). \end{equation}
 Now, for  any $j\in I_1(\mbf{q})$, according to Lemma \ref{thrm:QCLP}, $\tilde{\mbf{x}}_{\{i,j\}}$ must satisfy
\begin{equation}
\tilde{x}_i=\frac{q_i}{\|(q_i,q_j)^T\|_2}(\tilde{x}_i^2+\tilde{x}_j^2)^{1/2},
\label{eq:p_2-QCLP}
\end{equation}
where here we used the fact that $\rho^2-\sum_{l \neq i,j} \tilde{x}_l^2  = \tilde{x}_i^2+\tilde{x}_j^2$. Squaring both sides of \eqref{eq:p_2-QCLP}, we obtain that it is equivalent to $q_j^2 \tilde{x}_i^2 = q_i^2 \tilde{x}_j^2$,
and hence
$$
\tilde{x}_j^2 = \frac{q_j^2}{q_i^2}\tilde{x}_i^2  \qquad \mbox{ for any }j\in I_1( \mbf{q}).
$$
By (\ref{parta}), $\tilde{x}_j=0$ whenever $j \in I_0( \mbf{q})$, and we can therefore write
$$
\tilde{x}_j^2 = \frac{q_j^2}{q_i^2}\tilde{x}_i^2, \qquad   j=1,2,\ldots,n.
$$
Summing over $j=1,2,\ldots,n$, and using the fact that $\|\tilde{\mbf{x}}\|_2^2 = \rho^2$, it follows that
$$\sum_{j=1}^{n} \tilde{x}_i^2\frac{q_j^2}{q_i^2} = \rho^2,$$
implying that
$$
\tilde{x}_i^2 = \rho^2\frac{q_i^2}{\| \mbf{q}\|_2^2},
$$
which combined with the fact that that  $\text{sgn}(\tilde{x}_i)=\text{sgn}(q_i)$ (see \eqref{eq:p_2-QCLP} ), yields
$$
\tilde{x}_i = \rho\frac{q_i}{\|\mbf{q}\|_2}.
$$
Since we actually proved the latter for an arbitrary $i \in I_1(\mbf{q})$, and since $\tilde{x}_i = 0$ for any $i \in I_0(\mbf{q})$ (see (\ref{parta})), it follows that $$ \mbf{x} = \rho \frac{\mbf{q}}{\| \mbf{q}\|_2},$$
in contradiction to the assumption that $\tilde{\mbf{x}}$ is not an optimal solution of \eqref{eq:QCLP}.
\end{itemize} \qed
\end{proof}

The following corollary is a direct consequence of Lemmas \ref{thrm:S-QCLP} and \ref{thrm:2-QCLP}.

\begin{corollary}
\label{col:P_ij_sparse}
Let $\tilde{\mbf{x}}\in S$. If $\tilde{\mbf{x}}$ is \underline{not} an optimal solution to \eqref{eq:S-QCLP} and \linebreak $I_1(\tilde{\mbf{x}}) \subseteq T$ for some $T \in R_s(\mbf{p})$, then there exist indices $i,j \in T (i \neq j)$ such that $\tilde{\mbf{x}}_{\{i,j\}}$ is \underline{not} an optimal solution of \eqref{eq:2-QCLP}.
\end{corollary}
\begin{proof} Assume that $|T|=k$. Since $\tilde{\mbf{x}}$ is not an optimal solution of \eqref{eq:S-QCLP}, it follows by Lemma \ref{thrm:S-QCLP} that $\tilde{\mbf{x}}_T \neq \frac{\mbf{p}_T}{\|\mbf{p}_T\|}$, which implies that $\tilde{\mbf{x}}_T$ is not the optimal solution of the restricted problem
$$ \min_{\mbf{y} \in \mathbb{R}^k} \left \{ \mbf{p}_T^T \mbf{y} : \|\mbf{y}\|_2 \leq \rho \right \}.$$
Therefore, invoking  Lemma \ref{thrm:2-QCLP} with $d=k,\mbf{q} =\mbf{p}_T$, it follows that there exist indices $i,j \in T (i \neq j)$ such that $\tilde{\mbf{x}}_{i,j}$ is not an optimal solution of \eqref{eq:2-QCLP}.\qed
\end{proof}

\subsection{Co-Stationarity vs. CW-Maximality}
\label{subsec:co_stat_vs_cw_max}

The main result of this paper is given in the following theorem, which establishes the superiority of the CW-maximality condition over the co-stationarity condition.

\begin{theorem}
\label{thrm:CW_vs_Prop3.2}
Let $\mbf{x}$ be a CW-maximum point of problem \eqref{P}. Then, $\mbf{x}$ is a co-stationary point of \eqref{P}.
\end{theorem}

\begin{proof}
Let $\mbf{x}$ be a CW-maximum point of \eqref{P}. Assume by contradiction that $\mbf{x}$ is not a co-stationary point. This means that there exists a vector $\mbf{v}\in S$ such that
\begin{gather}
\nabla f(\mbf{x})^T (\mbf{v}-\mbf{x}) > 0.
\label{eq:CW_Prop_1}
\end{gather}
We will show that we can find a vector $\mbf{z}\in S_2(\mbf{x})$ such that
\begin{gather}
\nabla f(\mbf{x})^T (\mbf{z}-\mbf{x}) > 0.
\label{eq:CW_Prop_2}
\end{gather}
This will imply a contradiction to the CW-maximality of $\mbf{x}$ by the following simple argument: since $f$ is a convex function, we have
\begin{gather*}
f(\mbf{z}) \geq f(\mbf{x}) + \nabla f(\mbf{x})^T(\mbf{z}-\mbf{x}),
\end{gather*}
which combined with  \eqref{eq:CW_Prop_2} implies that
\begin{gather*}
f(\mbf{z}) > f(\mbf{x}),
\end{gather*}
which is an obvious contradiction to the CW-maximality of $\mbf{x}$.\\
\indent Since $\mbf{x}$ satisfies \eqref{eq:CW_Prop_1}, we obviously have $\nabla f(\mbf{x})\neq \mbf{0}$.
Let $X^*(\nabla f(\mbf{x}),s)$ be the set of optimal solutions of \eqref{eq:S-QCLP} with $\mbf{p}=\nabla f(\mbf{x})$ and let \linebreak $\mbf{x}^*\in X^*(\nabla f(\mbf{x}),s)$ be some particular solution. Then,
\begin{gather*}
\nabla f(\mbf{x})^T\mbf{x}^* \geq \nabla f(\mbf{x})^T \mbf{v} > \nabla f(\mbf{x})^T \mbf{x},
\end{gather*}
and thus $\mbf{x}\notin X^*(\nabla f(\mbf{x}),s)$.\\
Suppose that there exists some $l$ for which $\nabla_l f(\mbf{x})\cdot x_l < 0$ (and in particular $l\in I_1(\mbf{x})$). Define $\mbf{z}$ as:
\begin{gather*}
j = 1,\dots,n \qquad
z_j := \left\{ \begin{array}{ll}
-x_l, & j=l, \\
x_j, & otherwise.
\end{array} \right.
\end{gather*}
$\mbf{z}\in S_2(\mbf{x})$ and $\nabla f(\mbf{x})^T (\mbf{z}-\mbf{x}) > 0$ since
\begin{gather*}
\nabla f(\mbf{x})^T(\mbf{z}-\mbf{x}) = -2\cdot \nabla_l f(\mbf{x})\cdot x_l > 0.
\end{gather*}
We have thus shown in this case the desired contradiction.
From now on , we will therefore consider the case where $\nabla_i f(\mbf{x})\cdot x_i \geq 0$ for all $i=1,\dots,n$.\\
Consider the following cases:
\renewcommand{\labelenumi}{\arabic{enumi}.}
\renewcommand{\labelenumii}{\arabic{enumi}.\arabic{enumii}.}
\begin{enumerate}
\item  $ I_1(\mbf{x}) \not\subseteq I_{\geq}(\nabla f(\mbf{x}),s)$. \\
Obviously, there is some $h\in I_1(\mbf{x})\cap I_<(\nabla f(\mbf{x}),s)$. \\
We will consider the following subcases:
\begin{enumerate}
\item If $|I_\geq(\nabla f(\mbf{x}),s)|< s$, then $\nabla_h f(\mbf{x})=0$ (by Lemma \ref{lma:S-QCLP}, part 2), and since \mbox{$\nabla f(\mbf{x})\neq \mbf{0}$}, we conclude, using Lemma \ref{lma:S-QCLP} (part 1), that there is some $l\in I_\geq(\nabla f(\mbf{x}),s)$. Define $\mbf{z}$ as:
\begin{gather*}
j = 1,\dots,n \qquad
z_j := \left\{ \begin{array}{cl}
\text{sgn}\left(\nabla_l f(\mbf{x})\right)\cdot(x_h^2+x_l^2)^{1/2} & j=l, \\
0  & j = h, \\
x_j & otherwise.
\end{array} \right.
\end{gather*}
Obviously $\mbf{z}\in S_2(\mbf{x})$, and in addition $\nabla f(\mbf{x})^T (\mbf{z}-\mbf{x}) > 0$ since
\begin{gather*}
\nabla f(\mbf{x})^T(\mbf{z}-\mbf{x}) = \hspace{8cm}\\
\begin{array}{lll}
\hspace{1.5cm}
&   & \nabla_l f(\mbf{x})\cdot \text{sgn}\left(\nabla_l f(\mbf{x})\right)\cdot(x_h^2+x_l^2)^{1/2}-\nabla_l f(\mbf{x})\cdot x_l\\
& = &  \big|\nabla_l f(\mbf{x})\big|\cdot(x_h^2+x_l^2)^{1/2}-\nabla_l f(\mbf{x})\cdot x_l
 \\
& = &  \big|\nabla_l f(\mbf{x})\big|\cdot(x_h^2+x_l^2)^{1/2}-\big|\nabla_l f(\mbf{x})\big|\cdot |x_l|
\qquad \left(\nabla_l f(\mbf{x})\cdot x_l\geq 0\right)\\
& = & \big|\nabla_l f(\mbf{x}) \big|\cdot  \left((x_h^2+x_l^2)^{1/2}-|x_l|\right) > 0
\qquad\quad \left(\nabla_l f(\mbf{x}) \neq 0,x_h \neq 0\right).
\end{array}
\end{gather*}
\item If $|I_\geq(\nabla f(\mbf{x}),s)|\geq s$, then there is some $l\in I_\geq(\nabla f(\mbf{x}),s)$ such that \mbox{$l \notin I_1(\mbf{x})$}. Otherwise $I_\geq(\nabla f(\mbf{x}),s)\subseteq I_1(\mbf{x})$, and since \mbox{$|I_\geq(\nabla f(\mbf{x}),s)|\geq s$} and \mbox{$|I_1 (\mbf{x})|\leq s$}, we have that $I_\geq(\nabla f(\mbf{x}),s)=I_1 (\mbf{x})$, contradicting our assumption that \mbox{$I_1(\mbf{x}) \not\subseteq I_\geq(\nabla f(\mbf{x}),s)$}. We will define $\mbf{z}$ as:
\begin{gather*}
j = 1,\dots,n \qquad
z_j := \left\{ \begin{array}{cl}
\text{sgn}\left(\nabla_l f(\mbf{x})\right)\cdot |x_h| & j=l, \\
0  & j = h, \\
x_j & otherwise.
\end{array} \right.
\end{gather*}
Clearly, $\mbf{z}\in S_2(\mbf{x})$. In addition, $\nabla f(\mbf{x})^T (\mbf{z}-\mbf{x}) > 0$ since:
\begin{gather*}
\begin{array}{lllc}
\nabla f(\mbf{x})^T(\mbf{z} - \mbf{x}) & = &
\nabla_l f(\mbf{x}) \cdot \text{sgn}\left(\nabla_l f(\mbf{x})\right)\cdot |x_h| \\
&&- \nabla_h f(\mbf{x})\cdot x_h & \\
& =& \big|\nabla_l f(\mbf{x})\big|\cdot |x_h| - \big|\nabla_h f(\mbf{x})\big|\cdot |x_h| & \left(\nabla_h f(\mbf{x})\cdot x_h\geq 0\right)\\
& = &
\left(\big|\nabla_l f(\mbf{x})\big| - \big|\nabla_h f(\mbf{x})\big|
\right)
\cdot |x_h| >  0, &
\end{array}
\end{gather*}
where the last inequality holds since $x_h\neq 0$ and the indices $l$ and $h$ are such that $l\in I_\geq(\nabla f(\mbf{x}),s)$ and $h\in I_<(\nabla f(\mbf{x}),s)$, thus according to Lemma \ref{lma:S-QCLP} (part 3) $\big|\nabla_l f(\mbf{x})\big|>\big|\nabla_h f(\mbf{x})\big|$.
\end{enumerate}
\item $I_1(\mbf{x}) \subseteq I_\geq(\nabla f(\mbf{x}),s)$\\
Now we will consider the following subcases:
\begin{enumerate}
\item If $I_1(\mbf{x}) \subseteq T$ for some $T \in R_s(\nabla f(\mbf{x}))$, then since $\mbf{x} \notin X^*(\nabla f(\mbf{x}),s)$, it follows that according to Corollary \ref{col:P_ij_sparse}, there exist indices $h,l\in T$ such that
\begin{equation*}
\hat{\mbf{x}} := \argmax_{\inR{y}{2}} \lr{\{\}}{\nabla_{\{h,l\}} f(\mbf{x})^T \mbf{y}:\|\mbf{y}\|^2 \leq 1 - \sum_{i\neq h,l} x_i^2}
\end{equation*}
satisfies
\begin{gather}
\nabla_{\{h,l\}} f(\mbf{x})^T\hat{\mbf{x}} > \nabla_{\{h,l\}} f(\mbf{x})^T\mbf{x}_{\{h,l\}}.
\label{eq:2notopt}
\end{gather}
Since $|T|\leq s$ and $\|\hat{\mbf{x}}\|_2^2 \leq 1 - \sum_{i\neq h,l} x_i^2$, the vector
\begin{gather*}
j = 1,\dots,n \qquad
z_j := \left\{ \begin{array}{cl}
\hat{x}_1, & j=h, \\
\hat{x}_2, & j=l, \\
x_j, & otherwise,
\end{array} \right.
\end{gather*}
is in $S_2(\mbf{x})$, and satisfies by \eqref{eq:2notopt} that $\nabla f(\mbf{x})^T(\mbf{z}-\mbf{x})>0$.
\item If $I_1(\mbf{x}) \not\subseteq T$ for all $T \in R_s(\nabla f(\mbf{x}))$, then:
\begin{itemize}
\item Take $h\in I_1(\mbf{x})$ such that $h\notin T$ for some $T \in R_s(\nabla f(\mbf{x}))$. Since $I_1(\mbf{x})\subseteq I_{\geq}(\nabla f(\mbf{x}),s)$, it follows that $h\in I_{\geq}(\nabla f(\mbf{x}),s)$. Moreover, since  $I_{>}(\nabla f(\mbf{x}),s)\subseteq T$ and $h\notin T$, we have that $h\notin I_{>}(\nabla f(\mbf{x}),s)$, implying that $h\in I_{=}(\nabla f(\mbf{x}),s)$. Thus, $h\in I_{=}(\nabla f(\mbf{x}),s)\cap I_1(\mbf{x})$.
\item  $I_>(\nabla f(\mbf{x}),s)\not\subseteq I_1(\mbf{x})$. To show this, note that otherwise, \linebreak[4] $I_>(\nabla f(\mbf{x}),s)\subseteq I_1(\mbf{x})$, and since $I_1(\mbf{x}) \subseteq I_\geq(\nabla f(\mbf{x}),s)$ and \linebreak $|I_1(\mbf{x})|\leq s$, we obtain that  $| I_1(\mbf{x})|\leq \min\lr{\{\}}{s,|I_\geq(\nabla f(\mbf{x}),s)|}$, implying that $I_1(\mbf{x})\subseteq T$ for some $T\in R_s(\nabla f(\mbf{x}))$, in contradiction to our assumption.
Thus, there exists some $l\in I_>(\nabla f(\mbf{x}),s)$ such that $l\notin I_1(\mbf{x})$.
\end{itemize}
Define $\mbf{z}$ as:
\begin{gather*}
j = 1,\dots,n \qquad
z_j := \left\{ \begin{array}{cl}
\text{sgn}\left(\nabla_l f(\mbf{x})\right)\cdot |x_h|, & j=l, \\
0,  & j = h, \\
x_j, & otherwise.
\end{array} \right.
\end{gather*}
Clearly, $\mbf{z}\in S_2(\mbf{x})$. Furthermore, $\nabla f(\mbf{x})^T (\mbf{z}-\mbf{x}) > 0$ since
\begin{gather*}
\begin{array}{lllc}
\nabla f(\mbf{x})^T(\mbf{z} - \mbf{x}) & = &
\nabla_l f(\mbf{x})\cdot \text{sgn}\left(\nabla_l f(\mbf{x})\right)\cdot |x_h| & \\
&& - \nabla_h f(\mbf{x})\cdot x_h & \\
& =& \big|\nabla_l f(\mbf{x})\big|\cdot |x_h| - \big|\nabla_h f(\mbf{x})\big|\cdot |x_h| & \left(\nabla_h f(\mbf{x})\cdot x_h\geq 0\right)\\
& = &
\left(\big|\nabla_l f(\mbf{x})\big| - \big|\nabla_h f(\mbf{x})\big|
\right)
\cdot |x_h| > 0, &
\end{array}
\end{gather*}
where the last inequality holds since $x_h\neq 0$ and the indices $l$ and $h$ are such that $l\in I_>(\nabla f(\mbf{x}),s)$ and $h\in I_=(\nabla f(\mbf{x}),s)$, and thus according to Lemma \ref{lma:S-QCLP} (part 3) $\big|\nabla_l f(\mbf{x})\big|>\big|\nabla_h f(\mbf{x})\big|$.
\end{enumerate}
\end{enumerate}
We have thus arrived at a contradiction, and the desired implication is established. \qed
\end{proof}
In order to show that the reverse implication is not valid, that is, that co-stationary points are not necessarily CW-maximal points, we present an example of a problem instance and a co-stationary point, that is not a CW-maximal point.
\begin{example}
\label{exmpl:CW_vs_Prop3.2}
For any $n>s>0$, we consider problem \eqref{SPCA} with a diagonal matrix $\mbf{A}$, whose entries on the main diagonal are given by the vector $\mbf{a}$ defined by
\begin{gather*}
\mbf{a}:=\left(\begin{matrix} 2\cdot \mbf{1}_{n-s} \\ 0.5 \cdot \mbf{1}_{s} \end{matrix}\right),
\end{gather*}
where for a given positive integer $m$, $\mbf{1}_m$ and $\mbf{0}_m$ are the vectors of size $m$ with all entries equal to ones or zeros, respectively.
We also define
\begin{gather*}
\mbf{x}:=\left(\begin{matrix} \mbf{0}_{n-s} \\ s^{-0.5}\cdot \mbf{1}_{s} \end{matrix}\right) \qquad \text{and} \qquad
\tilde{\mbf{x}}:=\left(\begin{matrix} \mbf{0}_{n-s-1} \\ s^{-0.5} \\ 0\\ s^{-0.5}\cdot \mbf{1}_{s-1} \end{matrix}\right).
\end{gather*}
It easy to see that $\mbf{x},\tilde{\mbf{x}}\in S$ and that $\mbf{A}\succ {\mbf 0}$, since it is a diagonal matrix with positive diagonal elements.
The gradient of $f$ is given by:
\begin{gather*}
\nabla f(\mbf{x}) = 2\mbf{A}\mbf{x} = \left(\begin{matrix}
\mbf{0}_{n-s}\\
s^{-0.5}\cdot\mbf{1}_s
\end{matrix}\right).
\end{gather*}
For any $\mbf{v}\in S$:
\begin{gather*}
\langle \nabla f(\mbf{x}),\mbf{v} - \mbf{x} \rangle = \sum_{i=n-s+1}^{n} s^{-0.5}(v_i-s^{-0.5})  \\
= s^{-0.5} \left(\sum_{i=n-s+1}^{n} v_i -s^{0.5}\right) \leq s^{-0.5} \left(\|\mbf{v}\|_1 -s^{0.5}\right) \leq 0,
\end{gather*}
where the last inequality holds since $\|\mbf{v}\|_1\leq \sqrt{\|\mbf{v}\|_0}\|\mbf{v}\|_2 \leq\sqrt{s}$. Hence, $\mbf{x}$ is co-stationary. The vector $\tilde{\mbf{x}}$ satisfies $\tilde{\mbf{x}} \in S_2(\mbf{x})$ and since:
\begin{gather*}
\begin{array}{ll}
f(\tilde{\mbf{x}}) = \tilde{\mbf{x}}^T\mbf{A}\tilde{\mbf{x}} = (s-1)\cdot (2s)^{-1}+2s^{-1} &= (s+3)\cdot (2s)^{-1}\\
 &> s\cdot (2s)^{-1} = \mbf{x}^T\mbf{A}\mbf{x} = f(\mbf{x}),
\end{array}
\end{gather*}
it follows that $\mbf{x}$ is not a CW-maximum point.
\end{example}

\subsection{Support Optimality}
\label{subsec:SO}
Theorem \ref{thrm:CW_vs_Prop3.2} establishes the relationship between the two stationarity conditions considered up to this point:
co-stationarity and CW-maximality. A third condition, proposed in \cite{3-1_24}, that we will refer to as \emph{support optimality} (SO), is given in the following definition.
\begin{definition}[\bf{Support Optimality}]
\label{def:SO}
A vector $\mbf{x}^*\in S$ is called a {\bf support optimal (SO) point} of \eqref{P} with respect to an index set $T \subseteq \{1,2,\ldots,n\}$ if and only if it is an optimal solution of the optimization problem
\begin{equation}
\tag{SO}
\label{eq:def_SO}
\max_{\inR{x}{n}} \{ f(\mbf{x}):\|\mbf{x}\|_2 \leq 1,I_1(\mbf{x})\subseteq T\}.
\end{equation}
\end{definition}
It is clear that, if $\mbf{x} \in S$ is an optimal solution of problem \eqref{P}, then it must be an SO point of \eqref{P} with respect to any index set $T$ satisfying $|T| \leq s$ and $I_1(\mbf{x}) \subseteq T$. In that respect, support optimality is a necessary optimality condition for problem \eqref{P}. It is a remarkably weak condition and cannot be used exclusively to derive a reasonable algorithm. Nevertheless, it is not totally futile. In order to enhance the performance, the CW-based algorithms that will be presented in Section \ref{sec:Algos} will produce a sequence of SO points, and
in Section \ref{sec:Num_Res}, we will adopt the variational re-normalization strategy suggested in \cite{3-1_24}, stating that for each sparse solution obtained by any technique, it is reasonable to replace this solution with the SO point that correspond to the same support.  \\
We will conclude this section with an example that demonstrates the potential benefit of employing algorithms that produce a point that satisfies stronger necessary optimality conditions. Consider the pit-prop data, which consists of 13 variables measuring various physical properties of 180 pitprops. This data set was suggested originally in \cite{3-1_17_Jeffers1967}, and since then was extensively used as a benchmark example for sparse PCA; see, for example, \cite{3-1_17,3-1_18,3-1_24}. The problem has 13 variables and we consider a sparsity level of $s=4$. Note that we can list all the ${13 \choose 4}=715$ SO points that correspond to index sets with exactly $4$ indices, and the optimal solution  must be one of these 715 points. Out of this set of points, 28 satisfy the co-stationarity condition and only 2 satisfy the CW-maximality condition. The following table presents the support sets of each of the co-stationarity points along with their function values.

\begin{table}[H]
    \caption{The supports of the co-stationary points for the pit prop data.}
    \vspace{0.5cm}
    \begin{minipage}{.5\linewidth}
      \centering
        \begin{tabular}{|c|c|c|c|}
\hline
\# & Support & CW-maximum & Value
\\
\hline
1  & \{1,2,9,10\} 	& *	& 2.937 \\
2  & \{1,2,7,10\} 	&  	& 2.883 \\
3  & \{1,2,7,9\}	&  	& 2.859 \\
4  & \{1,2,8,9\}	&  	& 2.797 \\
5  & \{1,2,8,10\}	&  	& 2.759 \\
6  & \{1,2,6,7\}	&  	& 2.697 \\
7  & \{2,7,9,10\}	&  	& 2.696 \\
8  & \{2,6,7,10\}	&  	& 2.592 \\
9  & \{1,6,7,10\}	&  	& 2.587 \\
10 & \{1,2,3,4\}	& *	& 2.563 \\
11 & \{7,8,9,10\}	&  	& 2.549 \\
12 & \{6,7,9,10\}	&  	& 2.522 \\
13 & \{6,7,10,13\}	&  	& 2.459 \\
14 & \{6,7,8,10\}	&  	& 2.444 \\
\hline
        \end{tabular}
    \end{minipage}%
    \begin{minipage}{.5\linewidth}
      \centering
        \begin{tabular}{|c|c|c|c|}
\hline
\# & Support & CW-maximum & Value \\
\hline
15 & \{5,6,7,10\}	&  	& 2.337 \\
16 & \{7,8,10,12\}	&  	& 2.314 \\
17 & \{7,8,10,13\}	&  	& 2.302 \\
18 & \{5,6,7,13\}	&  	& 2.28 	\\
19 & \{3,4,6,7\}	&  	& 2.209 \\
20 & \{4,5,6,7\}	&  	& 2.196 \\
21 & \{7,10,12,13\}	&  	& 2.136 \\
22 & \{3,4,8,12\}  	&  	& 1.995 \\
23 & \{3,4,10,12\}	&  	& 1.992 \\
24 & \{3,10,11,12\}	&  	& 1.609 \\
25 & \{3,5,12,13\} 	&  	& 1.516 \\
26 & \{1,5,12,13\} 	&  	& 1.414 \\
27 & \{2,5,12,13\} 	&  	& 1.408 \\
28 & \{3,5,11,13\} 	&  	& 1.382 \\
\hline
        \end{tabular}
    \end{minipage}
\label{table:PitProp}
\end{table}
Since the number of CW-maximum points is significantly smaller than the number of co-stationary points, it is much more probable that the optimal solution will be found by an algorithm that produces CW-maximum points than an algorithm that produces co-stationary points.

\section{Algorithms}
\label{sec:Algos}

In this section, we will present two CW-based algorithms  -- GCW and PCW -- that are guaranteed to converge after a finite amount of iterations to a a CW-maxima. Later on, in Section \ref{sec:Num_Res}, we will demonstrate the superiority of these algorithm over methods which are based on the co-stationarity optimality condition such as the conditional gradient algorithm with unit step-size (ConGradU), that was suggested in \cite{3-1}, where it was also proven that  limit points of the sequence generated by ConGradU are co-stationary point.

In \cite{6-4} several algorithms that produce a CW-minimum point were considered. These block coordinate descent type algorithms perform at each iteration an optimization step with respect to one or two variables, while keeping the rest fixed. The coordinates that need to be altered are chosen to be the ones that produce the maximal decrease among all possible alternatives, or by applying an index selection strategy based on a local first order information. We adopt this approach and present similar algorithms for the sparse PCA problem.\\
At each iteration of a CW-based  algorithm applied to \eqref{P}, at most two variables will be updated. We can categorize each of the iterations according to whether the support is altered or not.
Block coordinate algorithms suffer from a major drawback -- a slow convergence rate. In order to reduce the effect of this displeasing characteristic, we will replace the point obtained at each step with an SO point that corresponds to the same support. This modification allows us to bypass the large amount of iterations that should have been devoted for optimizing the variables with respect to a fixed support. \\
\indent Below we present the Greedy CW (GCW) algorithm. We denote by $\mathcal{O}(T)$ an oracle that produce an SO point with respect to a given support $T$ by solving problem \eqref{eq:def_SO}. We will refer to this oracle as an \textit{SO oracle}. In the specific case of the PCA problem, the SO oracle amounts to finding a normalized principal eigenvector of a submatrix of the covariance matrix. However, finding the maximum of a general convex function $f$ over a unit ball is in principle a difficult task. We will assume that the solution produced by the oracle is uniquely defined by $T$. In addition, note that the oracle outputs an optimal solution of a problem consisting of maximizing a convex function over a compact and convex feasible set, and hence by \cite[Corollary 32.3.2]{rockafellar1970}, there exists an optimal solution of the problem which is an extreme point. In particular, this means that we can assume without any loss of generality that the oracle outputs a vector with norm $1$. This assumption will made from now on.

\begin{framed}
\noindent
{\bf The Greedy CW (GCW) Algorithm}\\
{\bf Input:}  $f:\mathbb{R}^n \rightarrow \mathbb{R}$ -- convex function; $\mathcal{O}(\cdot)$ -- SO oracle; $s$ -- sparsity level.\\
{\bf Output:} $\mbf{x}$ - a CW-maximum point of \eqref{SPCA}. \\
{\bf Initialization:} Take $T \in \lr{\{\}}{1,2,\dots,n}$ such that $1 \leq |T|\leq s$ and set \mbox{$\mbf{x}^0 = \mathcal{O}(T)$} and $k=0$.\\
{\bf General step:} 
\begin{enumerate}
\item \label{algo:itm:a1_itm2}  While  $\|\mbf{x}^k\|_0<s$, compute
\begin{equation*}
j_k\in \argmax_{j\in I_0(\mbf{x}^k)} \left\{f(\mbf{z}):\mbf{z}=\mathcal{O}(I_1(\mbf{x}^k)\cup \{j\}) \right\},
\end{equation*}
If  $f(\mathcal{O}(I_1(\mbf{x}^k)\cup\{j_k\}))>f(\mbf{x}^k)$, then set
\begin{eqnarray*}
 \mbf{x}^{k+1} &=& \mathcal{O}(I_1(\mbf{x}^k) \cup \{j_k\}),\\
k &=& k+1,
\end{eqnarray*}
\begin{itemize}
\item[] and return to \ref{algo:itm:a1_itm2}; otherwise, go to \ref{algo:itm:a1_itm3}.
\end{itemize}
\item \label{algo:itm:a1_itm3} For every $i\in I_1(\mbf{x}^k)$ and $j\in I_0(\mbf{x}^k)$ compute
$$
f_{i,j} = \max_{\sigma\in\{-1,1\}} \lr{\{\}}{f\lr{()}{\mbf{x}^k-x_i^k\mbf{e}_i+\sigma|x_i^k|\mbf{e}_j}}.
$$
Let $(i_k,j_k) = \argmax\lr{\{\}}{f_{i,j}:i\in I_1(\mbf{x}^k), j\in I_0(\mbf{x}^k)}$. If $f_{i_k,j_k}>f(\mbf{x}^k)$, then set
$$
\begin{aligned}
\mbf{x}^{k+1} &=&& \mathcal{O}\left (\left (I_1(\mbf{x}^k) \setminus \{i_k\}\right) \cup \{j_k\}\right),\\
k &=&& k+1,
\end{aligned}
$$
\begin{itemize}
\item[] and return to \ref{algo:itm:a1_itm2}.
\end{itemize}
Otherwise, STOP and set ${\mbf x} \leftarrow {\mbf x}^{k+1}$.
\end{enumerate}
\end{framed}

Step \ref{algo:itm:a1_itm2} of the GCW algorithm is in fact the greedy forward selection method proposed in \cite{3-1_24}. Hence, in some sense, the GCW method is a generalization of this method, that does not terminate at the moment that a solution with a full support is obtained. However, from a more practical point of view, this resemblance is irrelevant due to the fact that, if the initial support satisfies $|T|=s$, then the condition $\|\mbf{x}^k\|_0<s$ will probably be false for all $k$ in any reasonable practical scenario. \\
The following theorem summarizes the key properties of the GCW algorithm.
\begin{theorem}
	\label{thrm:GCW_properties}
	Let $\{\mbf{x}^k \}$ be the sequence generated by the GCW algorithm. Then, the following statements hold.
	\begin{enumerate}
		\item[(i)] The sequence of function values $\{f(\mbf{x}^k)\}$ is monotonically increasing.
		\item[(ii)] The algorithm terminates after a finite amount of iterations.
		\item[(iii)]  At termination, the algorithm produces a CW maximum point.
	\end{enumerate}
\end{theorem}

\begin{proof}
	

Part $(i)$  follows immediately from the description of the GCW algorithm. Part $(ii)$ is a consequence of the monotonicity of the algorithm (part $(i)$)   and the fact that it only passes through SO points, from which there is only a finite number under the standing assumption that the solution produced by the oracle $\mathcal{O}(T)$ is uniquely defined by $T$.\\
To prove $(iii)$,
consider the following partition of $S_2(\mbf{x})$:
\begin{gather*}
\begin{array}{rl}
S_2(\mbf{x}) =& \{\mbf{z}:\|\mbf{z}-\mbf{x}\|_0\leq 2,\mbf{z}\in S\}\\
=& S_2^0(\mbf{x}) \cup S_2^1(\mbf{x}) \cup S_2^2(\mbf{x}),\\
\end{array}
\end{gather*}
where
\begin{gather*}
\begin{array}{rl}
S_2^0(\mbf{x}) =& \{\mbf{z} \in S:\|\mbf{z}-\mbf{x}\|_0\leq 2,I_1(\mbf{z})\subseteq I_1(\mbf{x}) \}\\
S_2^1(\mbf{x}) =& \{\mbf{z} \in S:\|\mbf{z}-\mbf{x}\|_0\leq 2,I_1(\mbf{z})=I_1(\mbf{x})\cup\{j\}, j\in I_0(\mbf{x})\} \\
S_2^2(\mbf{x}) = & \hspace{9cm}~
\end{array}\\
\qquad \{\mbf{z} \in S:\|\mbf{z}-\mbf{x}\|_0\leq 2,I_1(\mbf{z})=(I_1(\mbf{x}) \setminus \{i\}) \cup \{j\},i\in I_1(\mbf{x}),j\in I_0(\mbf{x}) \},
\end{gather*}
and assume that the algorithm produced the point $\bar{\mbf{x}}$. Since $\bar{\mbf{x}}$ is an SO point and $S_2^0(\bar{\mbf{x}})\subseteq \{\mbf{x}:\|\mbf{x}\|_2 \leq 1,I_1(\mbf{x})\subseteq I_1(\bar{\mbf{x}})\}$, it follows that $f(\bar{\mbf{x}})\geq f(\mbf{x})$ for any $\mbf{x}\in S_2^0(\bar{\mbf{x}})$. Now, note that the algorithm terminates only if after performing Step \ref{algo:itm:a1_itm3} we obtain that for any $i\in I_1(\bar{\mbf{x}})$ and $j\in I_0(\bar{\mbf{x}})$
\begin{gather*}
\def\arraystretch{1.5}
\begin{array}{ll}
f_{i,j} &=\max_{\sigma\in\{-1,1\}} \lr{\{\}}{f\lr{()}{\bar{\mbf{x}}-\bar{x}_i\mbf{e}_i+\sigma|\bar{x}_i|\mbf{e}_j}}\\
&= \max_{\alpha} \lr{\{\}}{f\lr{()}{\bar{\mbf{x}}-\bar{x}_i\mbf{e}_i+\alpha\mbf{e}_j}:\alpha\in [-|\bar{x}_i|,|\bar{x}_i|]}\\
&\leq f(\bar{\mbf{x}}),
\end{array}
\end{gather*}
where the first equality is due to the fact that the maximum of a convex function over a compact and convex set is attained at an extreme point, see \cite[Corolalry 32.3.2]{rockafellar1970}.  Thus, $f(\bar{\mbf{x}})\geq f(\mbf{x})$ for any $\mbf{x}\in S_2^2(\bar{\mbf{x}})$.
This is enough for proving that $\bar{\mbf{x}}$ is CW-maximal in the case when $\|\bar{\mbf{x}}\|_0=s$ since in this case $S_2^1(\bar{\mbf{x}})=\emptyset$. If $\|\bar{\mbf{x}}\|_0<s$, then prior to entering Step \ref{algo:itm:a1_itm3},  Step \ref{algo:itm:a1_itm2} must be performed. This step is terminated only if $f(\bar{\mbf{x}})\geq f(\mbf{x})$ for any $$\mbf{x}\in \{\mbf{z} \in S:I_1(\mbf{z})=I_1(\bar{\mbf{x}})\cup\{j\}, j\in I_0(\bar{\mbf{x}})\},$$ and since $S_2^1(\bar{\mbf{x}})\subseteq \{\mbf{z} \in S:I_1(\mbf{z})=I_1(\bar{\mbf{x}})\cup\{j\}, j\in I_0(\bar{\mbf{x}})\}$, it implies that $f(\bar{\mbf{x}})\geq f(\mbf{x})$ for any $\mbf{x}\in S_2^1(\bar{\mbf{x}})$, concluding that $f(\bar{\mbf{x}})\geq f(\mbf{x})$ for any \mbox{$\mbf{x}\in S_2(\bar{\mbf{x}})$}.
\qed	

\end{proof}

 Practically, if the initial support $T$ satisfies $|T|=s$, then most of the computation time in the GCW method is consumed in computing $f_{i,j}$ for each possible swap.
   This observation encourages us to consider the following variation of GCW, which we name \textit{the Partial CW (PCW)} algorithm.
\begin{framed}
\noindent
{\bf The Partial CW (PCW) Algorithm}\\
{\bf Input:}  $f:\mathbb{R}^n \rightarrow \mathbb{R}$ -- convex function; $\mathcal{O}(\cdot)$ -- SO oracle; $s$ -- sparsity level.\\
{\bf Output:} $\mbf{x}$ - a CW-maximum point of \eqref{SPCA}. \\
{\bf Initialization:} Take $T \in \lr{\{\}}{1,2,\dots,n}$ such that $1 \leq |T|\leq s$ and set \mbox{$\mbf{x}^0 = \mathcal{O}(T)$} and $k=0$.\\
{\bf General step:} 
\begin{enumerate}
\item \label{algo:itm:a2_itm2}  While  $\|\mbf{x}^k\|_0<s$,  compute
\begin{equation*}
j_k\in \argmax_{j\in I_0(\mbf{x}^k)} \left\{f(\mbf{z}):\mbf{z}=\mathcal{O}(I_1(\mbf{x}^k)\cup \{j\}) \right\},
\end{equation*}
If  $f(\mathcal{O}(I_1(\mbf{x}^k)\cup\{j_k\}))>f(\mbf{x}^k)$, then set
\begin{eqnarray*}
\mbf{x}^{k+1} & =& \mathcal{O}(I_1(\mbf{x}^k) \cup \{j_k\}),\\
k &=& k+1,
\end{eqnarray*}
\begin{itemize}
\item[] and return to \ref{algo:itm:a1_itm2}; otherwise, go to \ref{algo:itm:a2_itm3}.
\end{itemize}
\item \label{algo:itm:a2_itm3} Set $R=I_1(\mbf{x}^k)$.\\
While $|R|>0$
\begin{itemize}
\item[] Set $i_k \in \argmin\lr{\{\}}{|x^k_i|:i\in R}$ and for each $j\in I_0(\mbf{x}^k)$ compute
\begin{gather*}
f_{i_k,j} = \max_{\sigma\in\{-1,1\}} \lr{\{\}}{f\lr{()}{\mbf{x}^k-x_{i_k}^k\mbf{e}_{i_k}+\sigma|x_{i_k}^k|\mbf{e}_j}}.
\end{gather*}
Let $j_k \in \argmax\lr{\{\}}{f_{i_k,j}: j\in I_0(\mbf{x}^k)}$. \\
If $f_{i_k,j_k}>f(\mbf{x}^k)$, then set
\begin{eqnarray*}
 \mbf{x}^{k+1} &=&  \mathcal{O}\left (\left (I_1(\mbf{x}^k) \setminus \{i_k\}\right) \cup \{j_k\}\right),\\
k &=& k+1,
\end{eqnarray*}
\begin{itemize}
\item[] and return to \ref{algo:itm:a2_itm2}.
\end{itemize}
Otherwise, set $R=R\setminus \{i_k\}$.
\end{itemize}
STOP and set ${\mbf x} \leftarrow {\mbf x}^{k+1}$.
\end{enumerate}
\end{framed}

Before termination, PCW will perform the computation of all possible $f_{i,j}$, thus assuring the convergence to a CW-maximum point, given that the output is of a full support.
For the general step, the amount of computation will significantly decrease on the expense of finding the indices that provide the maximal increase in the function value. Nevertheless, the empirical study suggests that PCW provides similar results as GCW with respect to function values in a fraction of the time, as demonstrated in Section \ref{sec:Num_Res}.\\

\section{Numerical Results}
\label{sec:Num_Res}

We will illustrate the effectiveness of the algorithms proposed in the previous section on simulated and a gene expression datasets.
We compared the results with the following alternative algorithms: the novel $l_0$-constrained version of ConGradU \cite{3-1}, the expectation maximization \cite{3-1_29}, approximate greedy \cite{3-1_10} and thresholding \cite{3-1_7}. The MATLAB implementation of ConGradU was kindly provided by the authors, for all the other alternative algorithms we used a MATLAB implementation available on the authors' web-pages. For the thresholding algorithm and the algorithms proposed in this paper, we used a MATLAB implementation, which is available in the following URL:\\
\centerline{\url{http://tx.technion.ac.il/~yakovv/packages/CW_PCA.zip}}\\
Whenever an initialization is required, we set the initial point to be the solution of the thresholding method.  Regarding the output, we adopt the variational renormalization strategy suggested in \cite{3-1_24}. Hence, for each of the algorithms, we extracted the sparsity pattern (the set of indices of the nonzero elements). The actual output vector is determined  to be equal to $\mathcal{O}(T)$, where $T$ is the generated sparsity pattern. The experiments were conduced on a PC with a 3.40GHz processor with 16GB RAM.

\subsection{Random Data}
\label{subsec:rand_data}
The covariance matrix $\mbf{A}$ is given by $\mbf{A}=\mbf{D}^T \mbf{D}$, where $\mbf{D}$ is the so-called "data matrix". Each entry in the data matrix $\inR{D}{m\times n}$ was randomly generated according to the Gaussian distribution with zero mean and variance $1/m$ ($D_{i,j}\sim \mathcal{N}(0,1/m)$). We considered data matrices with $n=2000,5000,10,000$ and $50,000$ variables. The number of observations is set to $m=150$ for all matrices. The sparsity levels  considered are $s=5,10,\dots,250$, and for each sparsity level we generated $100$ realizations. We will measure the effectiveness of the algorithms according to the average proportion of variability explained by the algorithm with respect to the largest eigenvalue of the data covariance matrix (i.e., $\mbf{x}^T\mbf{A}\mbf{x}/\lambda_1(\mbf{A})$, where $\mbf{x}$ is the solution and $\lambda_1(\mbf{A})$ is the largest eigenvalue of $\mbf{A}$).

\subsubsection{GCW vs. PCW}
\label{subsubse:GCW_vs_PCW}
First, we would like to compare the effectiveness and performance of the CW-based algorithms proposed in the previous section: GCW and PCW. We conducted the comparison based on data matrices with $2,000$ variables and the
results are given in Figure \ref{NumRes:fig:G_Vs._P}.
\begin{figure}[H]
\leavevmode
\centering
\includegraphics[angle=0,width=0.45\textwidth]{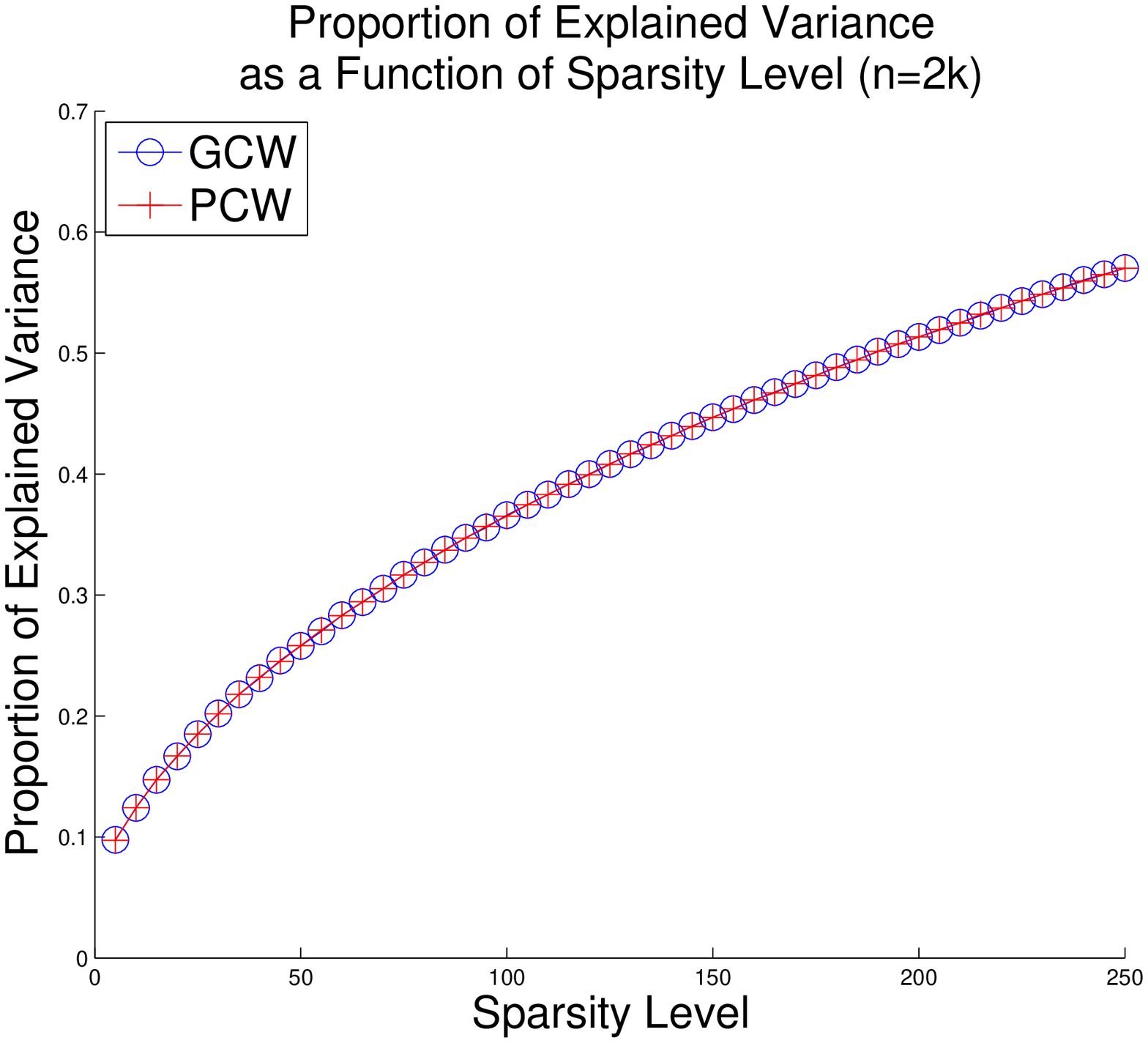}
\includegraphics[angle=0,width=0.45\textwidth]{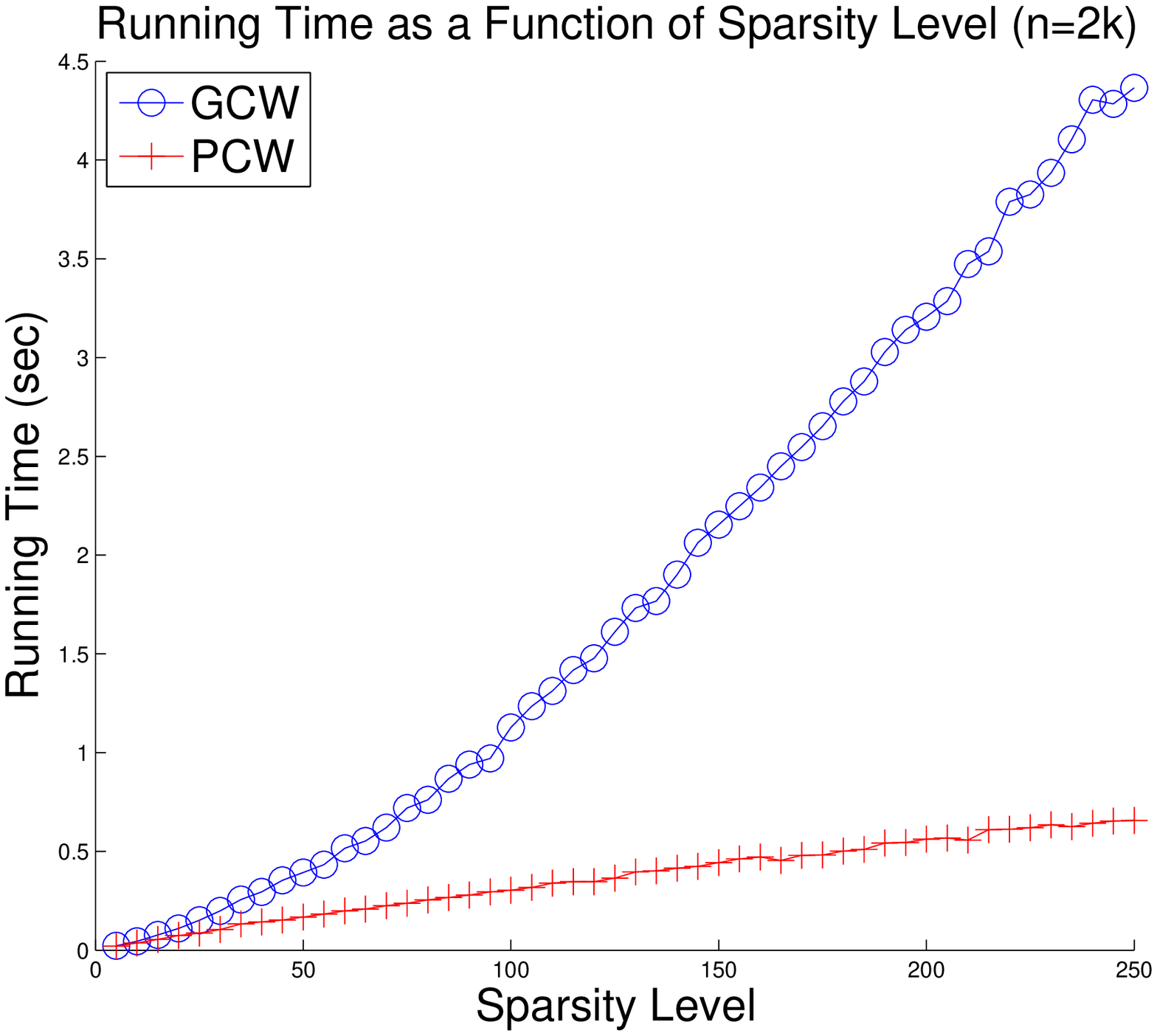}
\caption{\emph{GCW vs. PCW - The proportion of explained variability is given in the left figure and the computation time is given in the right one. The plot in both figures are given as a function of the sparsity level.}}
\label{NumRes:fig:G_Vs._P}
\end{figure}

We can clearly see that both methods achieve similar results with respect to the function values, while PCW achieves these results in a fraction of the time.
Thus, in the remaining numerical study we will omit GCW. Although the partial version remarkably reduces the computation time, it is still not competitive for  very large-scale problems when a full path of solutions is required. Thus, for such cases, we will also examine the effect of initializing PCW with the solution of the previous run (with the smaller sparsity level), and we will refer to such a continuation scheme as $\text{PCW}_{cont}$.

\subsubsection{PCW vs. Alternative Methods}
\label{subsubsec:PCW_vs_altern}
We will now compare the effectiveness and performance of PCW with respect to the alternative algorithms mentioned earlier. The setting for this set of experiments is the same as the one described in the previous example, but with problems with $n=5,000,~10,000$ and $50,000$ variables. Figure \ref{NumRes:fig:P_Vs._Ot_fval} provides the proportion of explained variability as a function of the sparsity level.

\begin{figure}[H]
\leavevmode
\centering
\includegraphics[angle=0,width=0.4\textwidth]{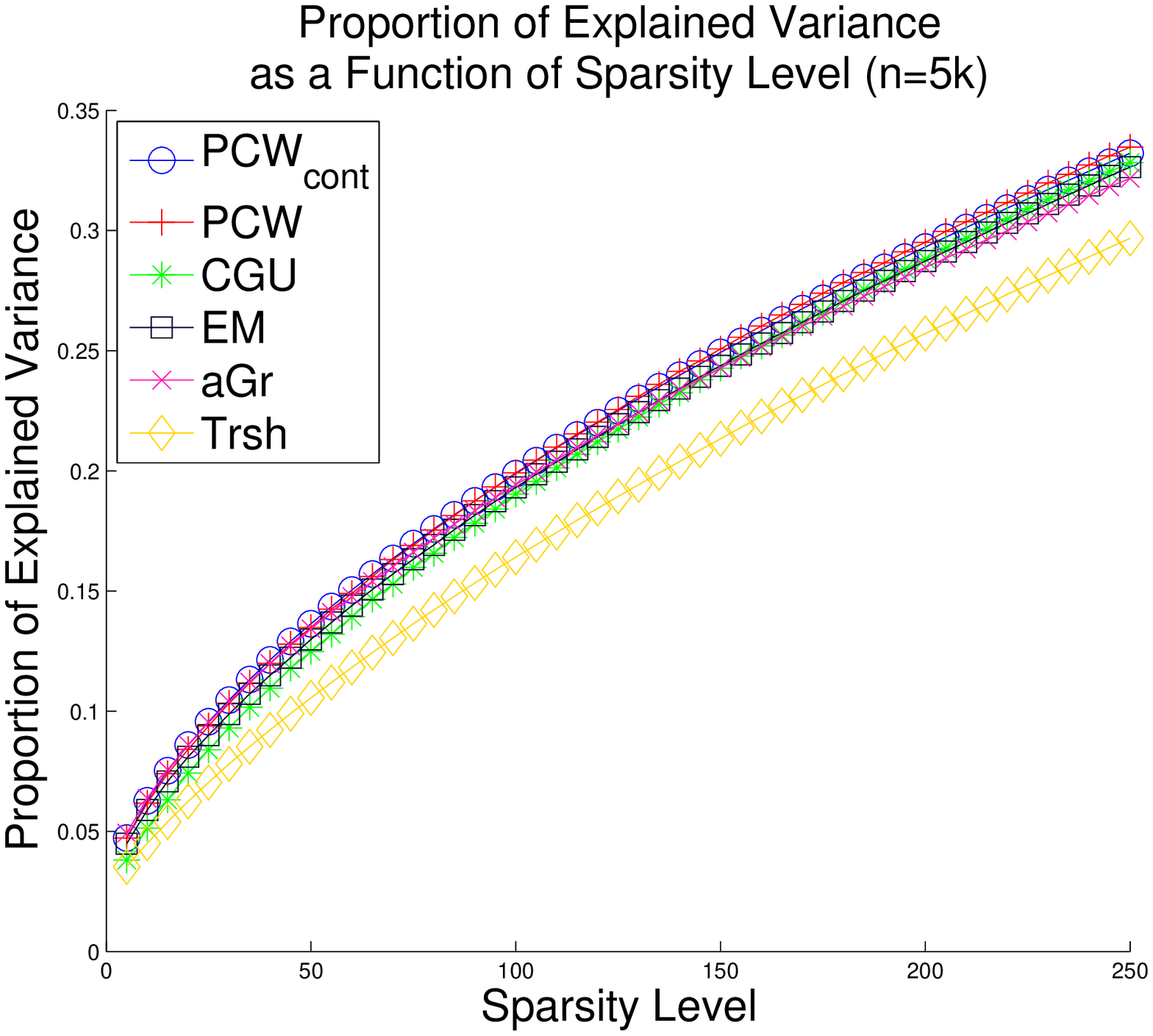}
\includegraphics[angle=0,width=0.4\textwidth]{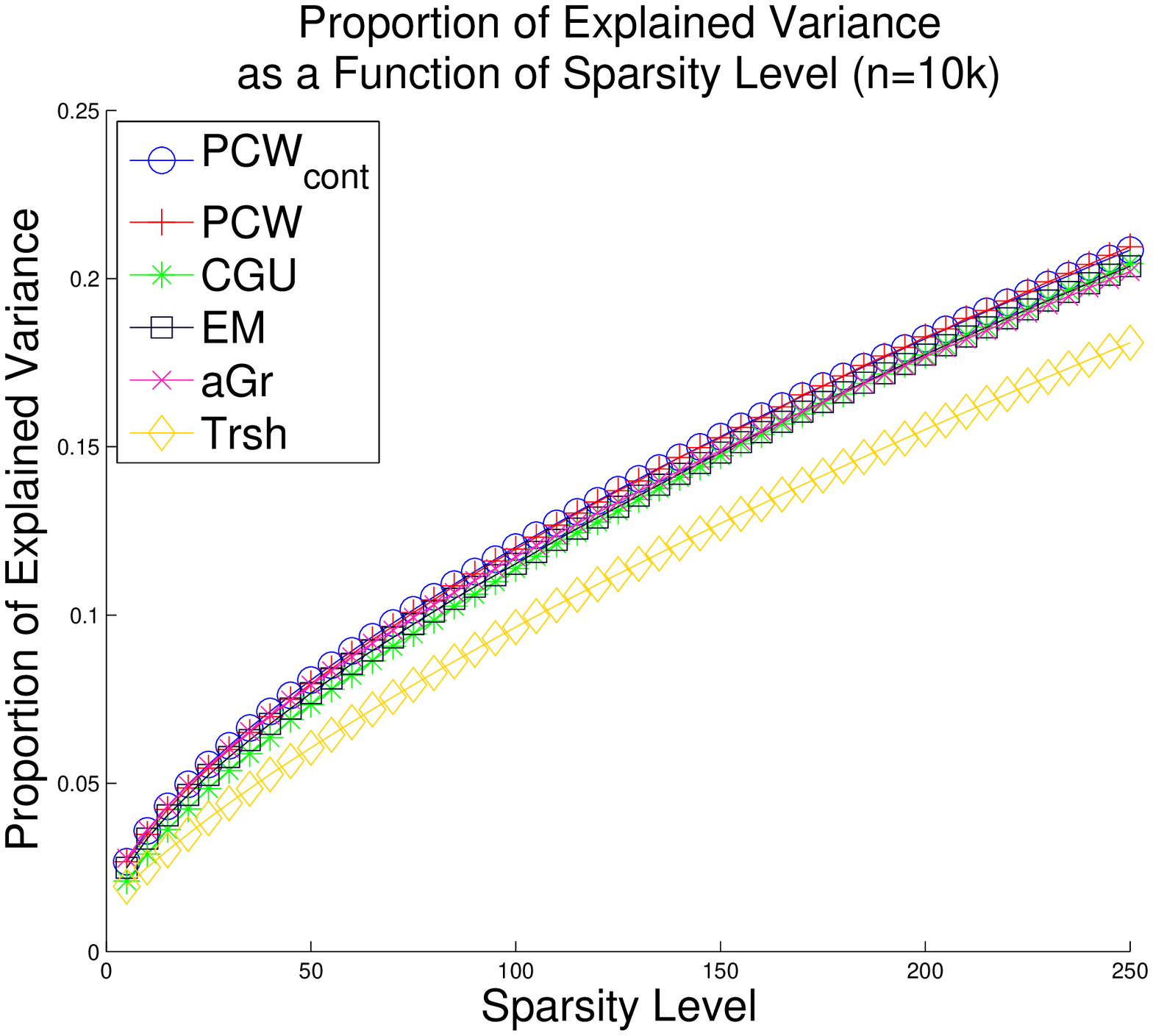}
\includegraphics[angle=0,width=0.4\textwidth]{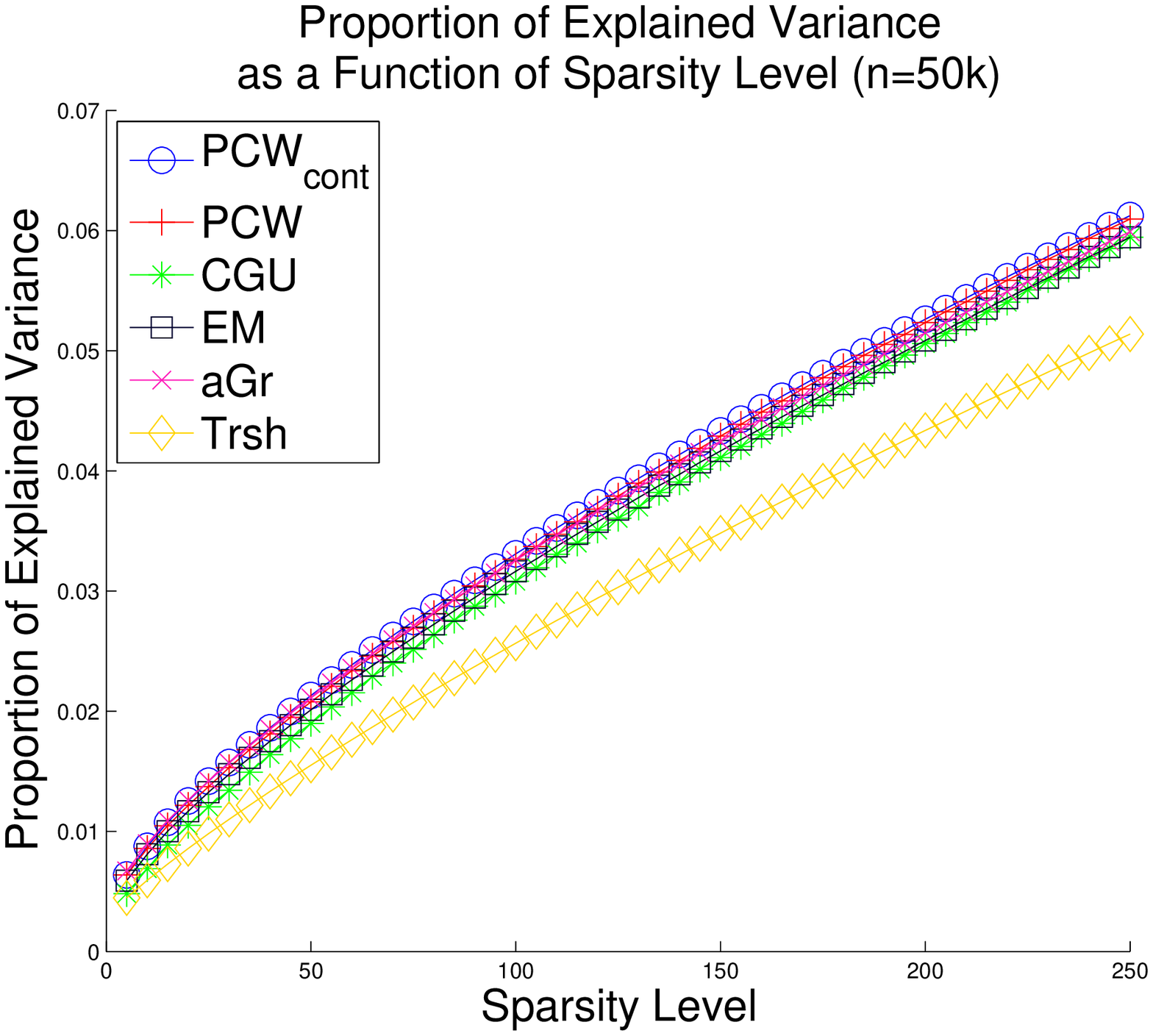}
\caption{\emph{PCW  vs. Others - The proportion of explained variability as a function of the sparsity level for $n=5000,10,000$ and $50,000$ are given in the upper left, upper right and bottom figures, respectively.}}
\label{NumRes:fig:P_Vs._Ot_fval}
\end{figure}
For small sparsity levels ($<50$) most of the algorithms provide similar results, but as the sparsity level is increased, the CW algorithms becomes superior to all the other methods. This advantage is not achieved without a price. In Figure \ref{NumRes:fig:P_Vs._Ot_telap} we provide the cumulative computation time of the algorithms (the cumulative time is considered since the approximate greedy algorithm provides a full set of solutions).

\begin{figure}[H]
\leavevmode
\centering
\includegraphics[angle=0,width=0.4\textwidth]{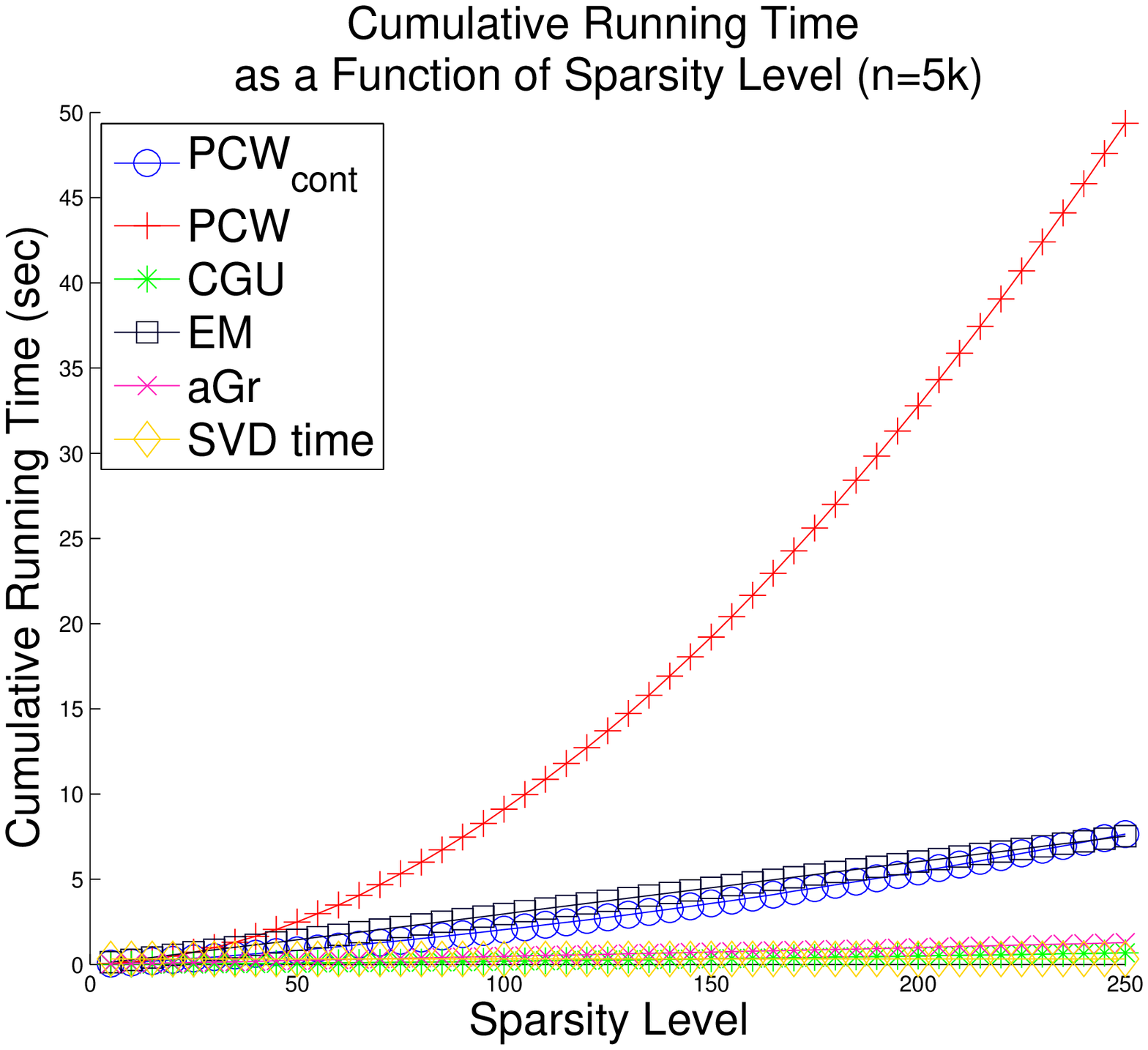}
\includegraphics[angle=0,width=0.4\textwidth]{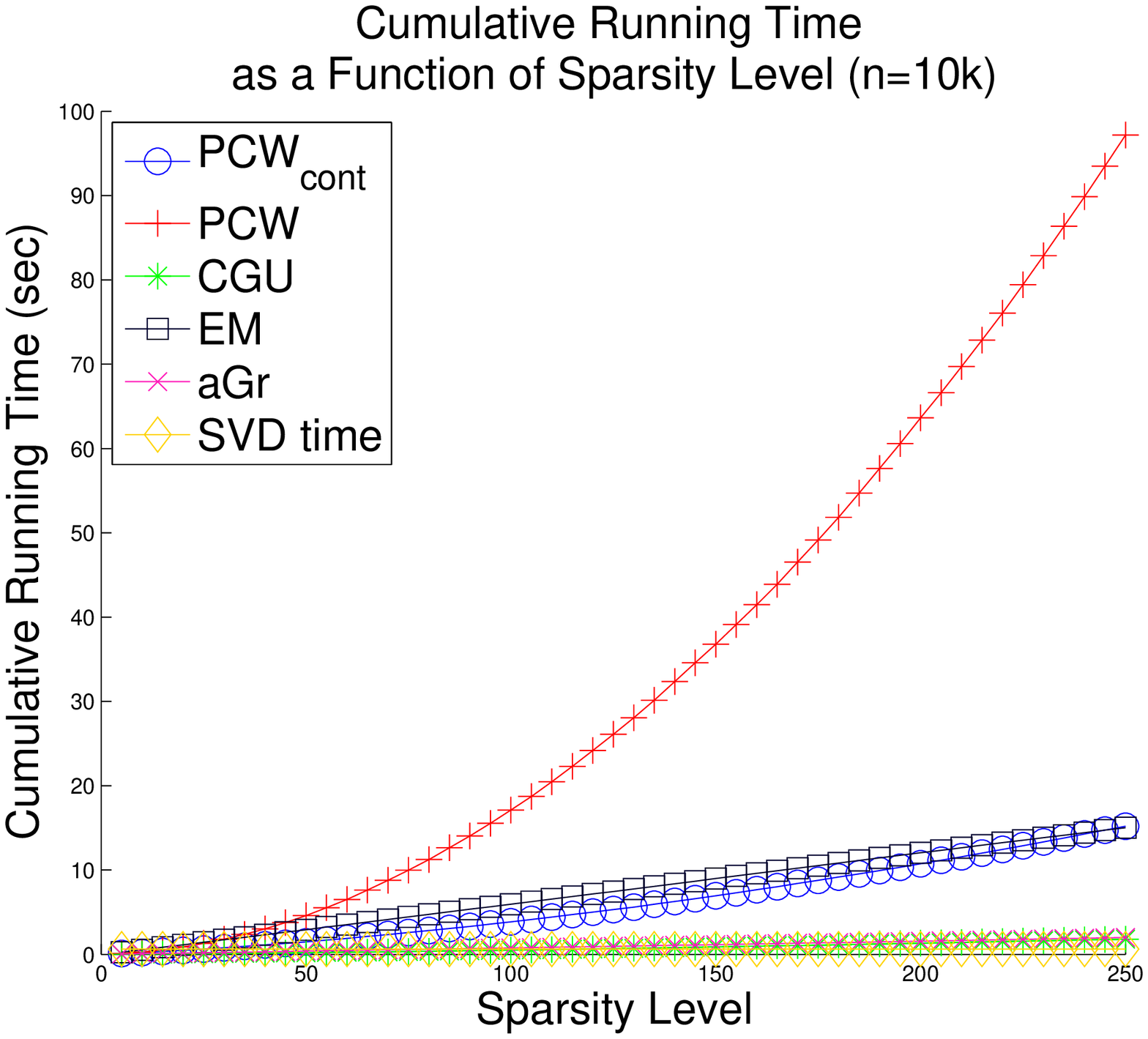}
\includegraphics[angle=0,width=0.4\textwidth]{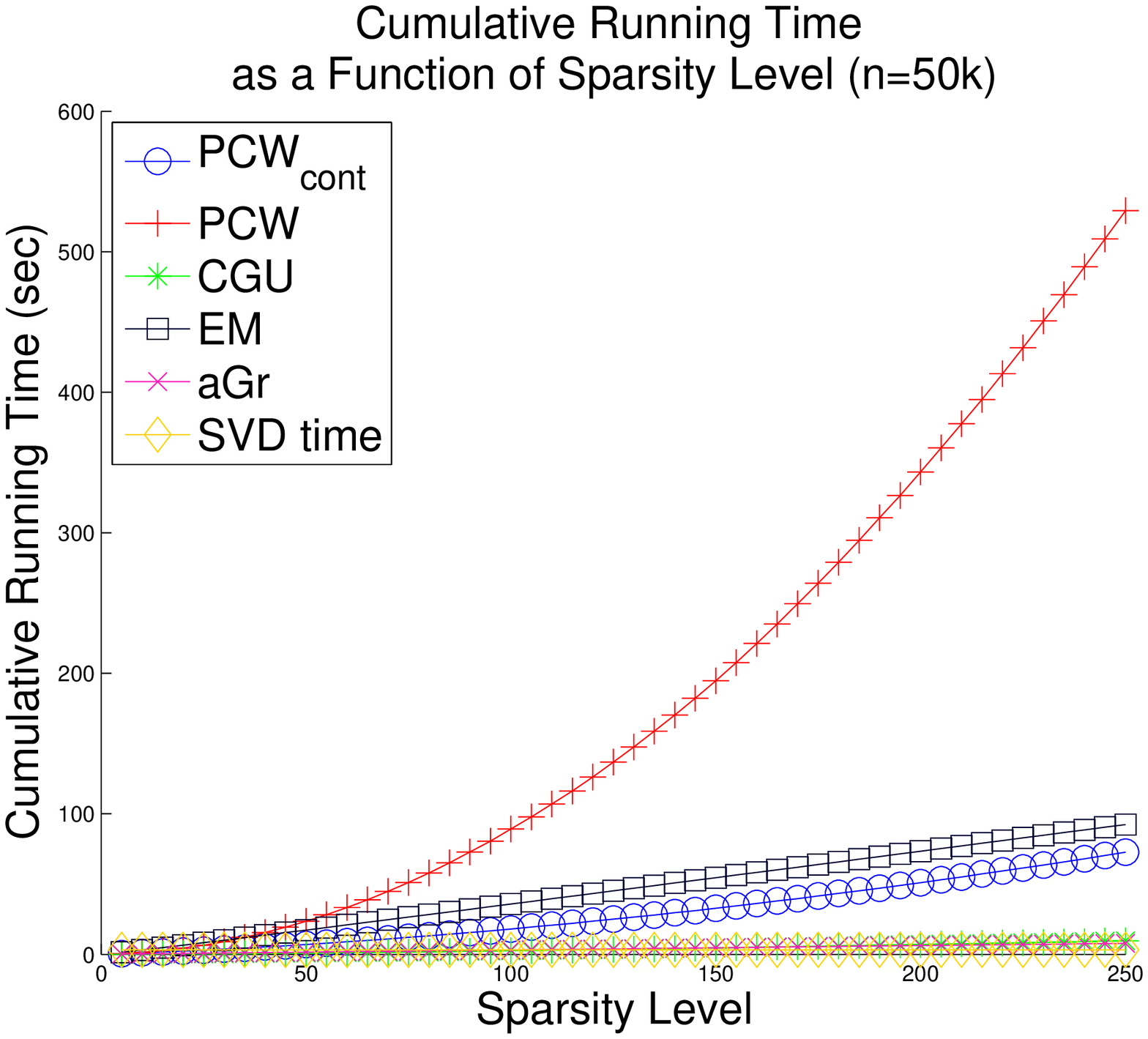}
\caption{\emph{PCW  vs. Alternative Methods - The cumulative computation times as a function of the sparsity level for $n=5,000,10,000$ and $50,000$ are given in the upper left, upper right and bottom, respectively. The SVD time is the time required for computing the principal eigenvector of the covariance matrix that corresponds to the generated data, which is used in order to find the thresholding solution, and in order to initialize the CW and ConGradU algorithms.}}
\label{NumRes:fig:P_Vs._Ot_telap}
\end{figure}

Even though PCW has greatly decreased the computation time with respect to GCW, it still requires a notably higher amount of computation time with respect to the alternative algorithms. The scheme we referred as $\text{PCW}_{cont}$ achieves similar results to PCW with respect to the function value. Regarding the running time, this scheme is competitive to the EM algorithm and requires somewhat more computational effort than the ConGradU and approximate greedy algorithms, thus providing a reasonable approach when a full set of solutions is required.

\subsection{Gene Expression Dataset}
\label{subsec:gene_data}
Sparse PCA is extensively utilized in the identification of the genes that reflect the changes in the gene expression patterns during different biological states, thus contributing to the diagnosis and research of certain diseases such as cancer. Figure \ref{NumRes:fig:GE} illustrates the proportion of explained variability and the cumulative running time for a Leukemia data set \cite{3-1_29_Armstrong2002}. This data set is composed from gene expression profiles of 72 patients with 12582 genes. The data set is normalized such that it has zero mean and unit variance.
\begin{figure}[H]
\leavevmode
\centering
\includegraphics[angle=0,width=0.45\textwidth]{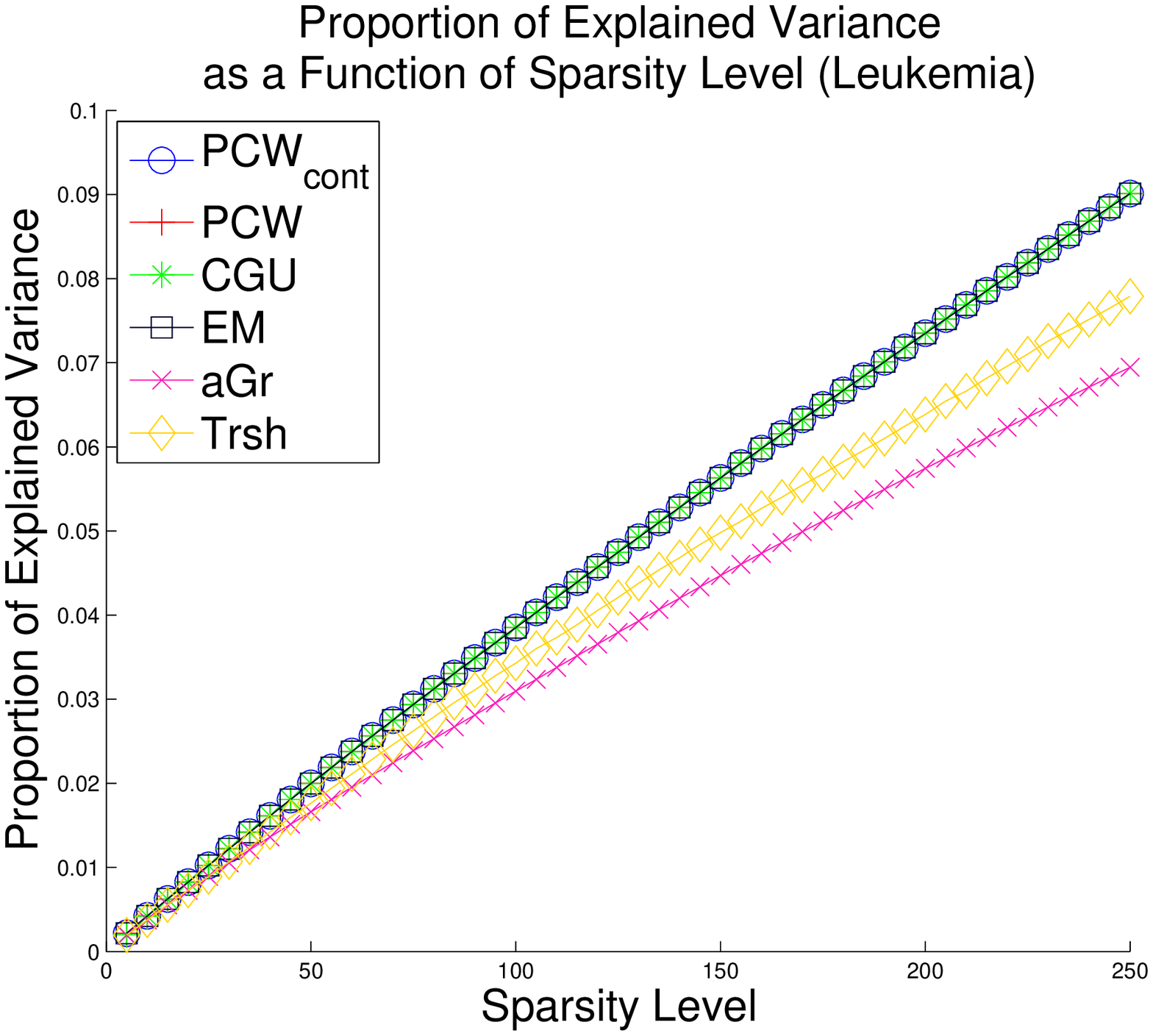}
\includegraphics[angle=0,width=0.45\textwidth]{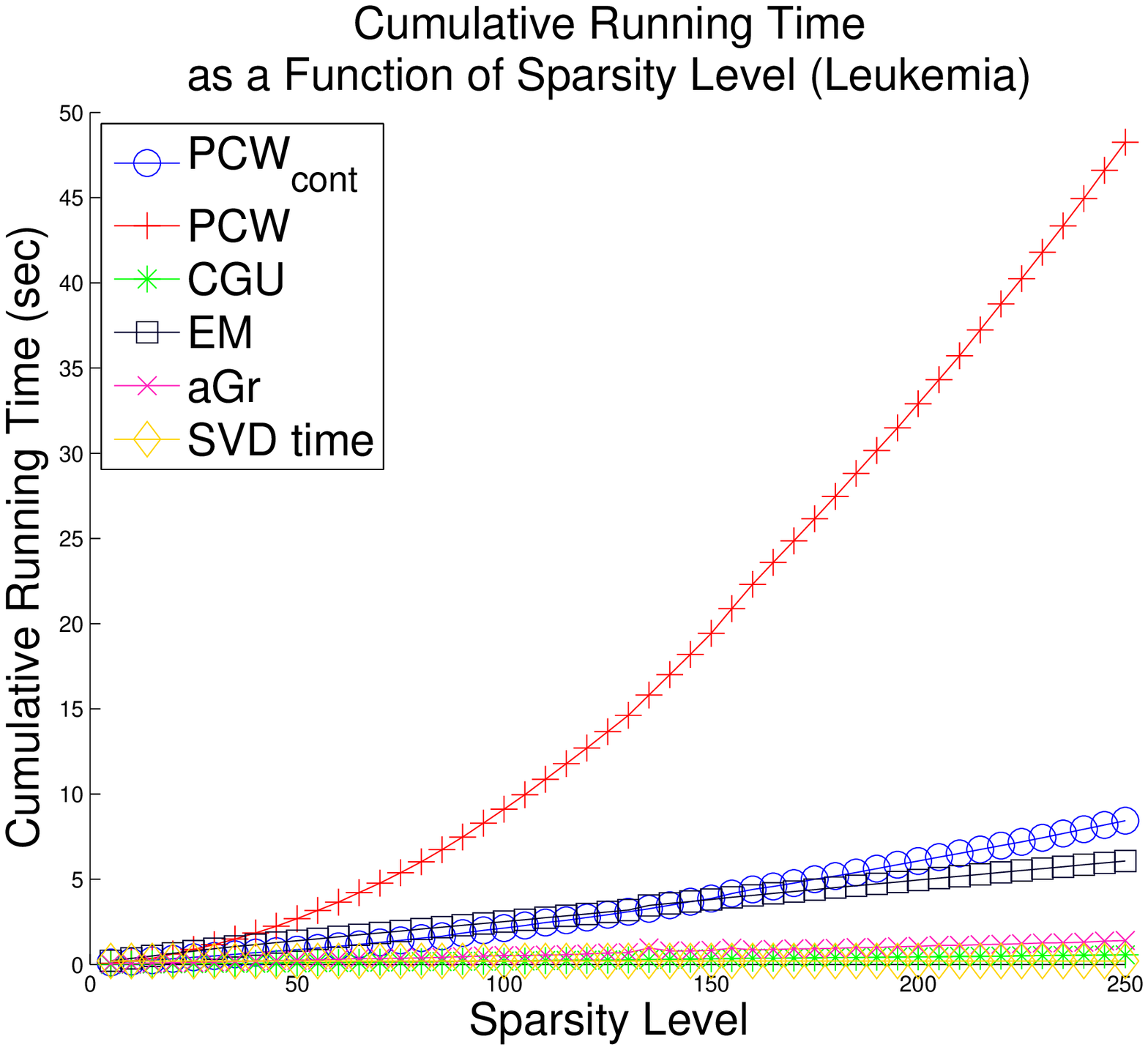}
\caption{\emph{Leukemia Gene Expression Data - The proportion of explained variability is given in the left figure and the cumulative computation time is given in the right one. The plot in both figures are given as a function of the sparsity level.}}
\label{NumRes:fig:GE}
\end{figure}

Most of the algorithms under consideration provide similar results with respect to the explained variability, which might indicate that this problem is, in a sense, rather easy to solve. We conducted similar experiments for additional 20 gene expression data sets from the GeneChip oncology database \cite{8-2} that is publicly available in:\\
\centerline{\url{http://compbio.dfci.harvard.edu/compbio/tools/gcod}}\\
while commonly, all the algorithms provided similar results, we can still see in Figure \ref{NumRes:fig:GEmany} that PCW yields the best solution (with respect to the function value) more times than the alternative algorithms, and consequently it obtains the smallest mean error with respect to the best solution.

\begin{figure}[H]
\leavevmode
\centering
\includegraphics[angle=0,width=0.45\textwidth]{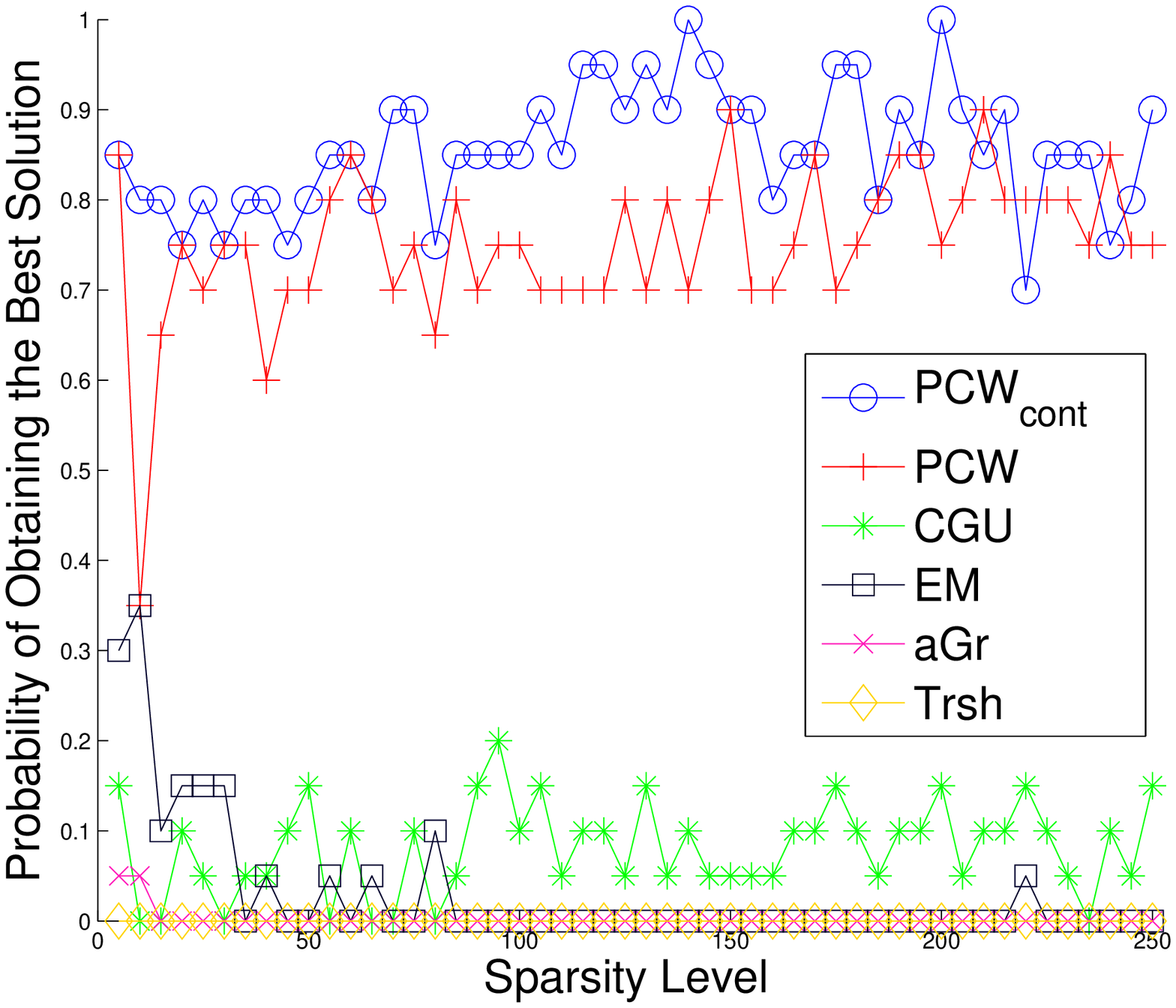}
\includegraphics[angle=0,width=0.45\textwidth]{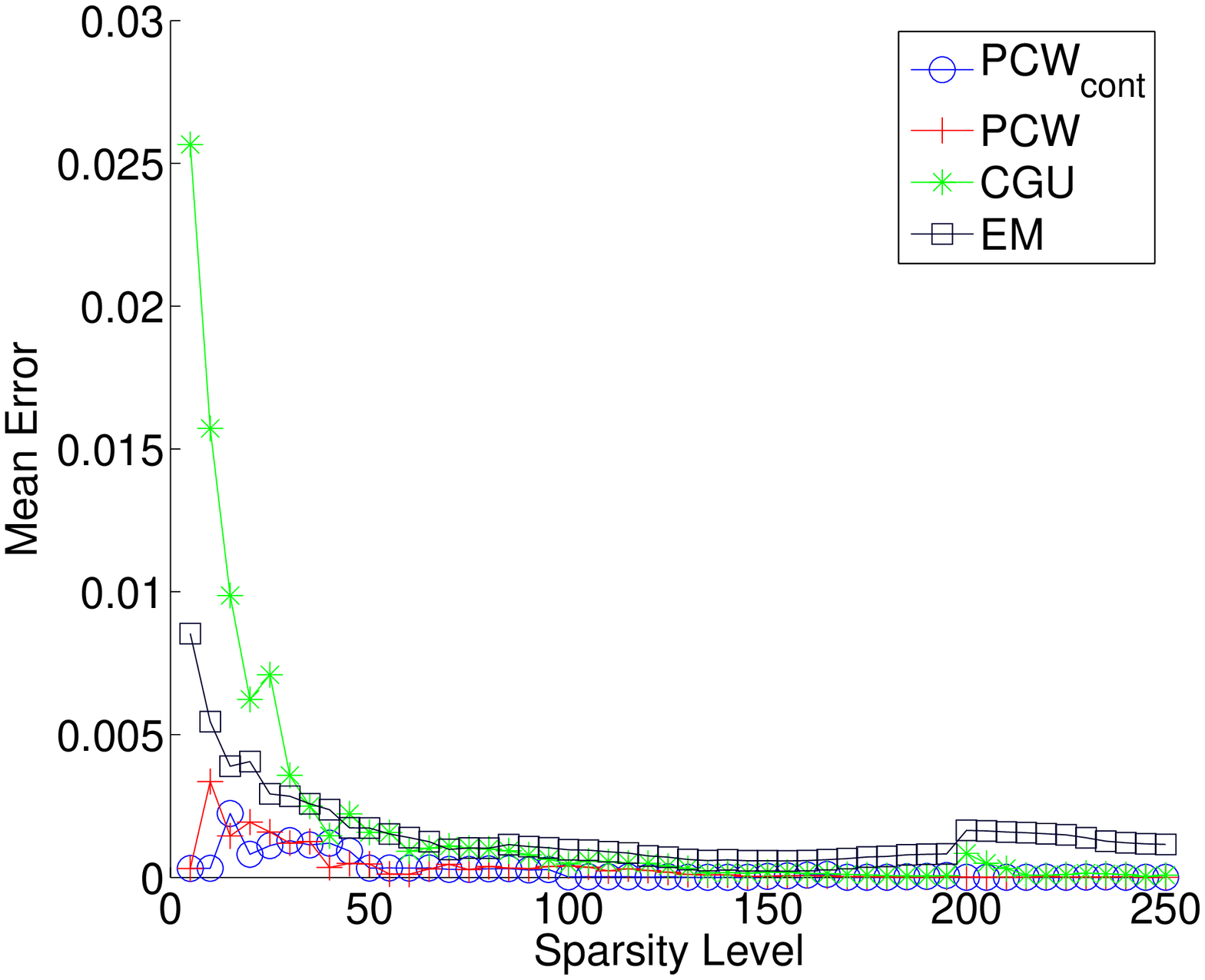}
\caption{\emph{Gene Expression Data - The left figure illustrates for each sparsity level the proportion of the number of data sets for which each algorithm obtained the best solution. The right figure illustrates for each sparsity level the mean error with respect to the best solution (the approximate greedy and thresholding algorithms were disregarded since both of them provide relative poor results).}}
\label{NumRes:fig:GEmany}
\end{figure}
	\section{Conclusions}
	In this paper, we considered the problem of maximizing a continuously differentiable convex function over the intersection of an $l_2$ unit ball and a sparsity constraint.   We have shown that coordinate-wise maximality is a more restrictive condition than  co-stationarity, which is the basis of many well-known methods for solving the sparse PCA problem. We introduced two algorithms (GCW and PCW) that are guaranteed to produce a CW-maximal solution, and demonstrated empirically the potential benefit of using this algorithms over some common algorithms proposed for this problem.

\begin{acknowledgements}
We would like to thank two anonymous referees for their helpful remarks that helped to improve the presentation of the paper.
\end{acknowledgements}
\bibliographystyle{spmpsci_unsrt}


\bibliography{PCA}


\end{document}